\documentclass[11pt]{amsart}

\newcommand{\hide}[1]{}

\usepackage{amssymb}
\usepackage{graphicx}

\def\textcolor#1{}

\newcommand{\N}{\mathbb{N}}
\newcommand{\R}{\mathbb{R}}
\newcommand{\C}{\mathbb{C}}

\renewcommand{\S}{\mathbb{S}}
\newcommand{\disk}{\mathbb{D}}
\newcommand{\Cc}{\widehat{{\C}}}
\newcommand{\Circle}{{\mathbb S}^1}
\newcommand{\Sphere}{{\mathbb S}^2}
\newcommand{\Deltahat}{\hat\Delta}

\renewcommand{\tilde}{\widetilde}
\renewcommand{\rho}{\varrho}
\renewcommand{\phi}{\varphi}

\renewcommand{\theta}{\vartheta}

\def\ol{\overline}
\newcommand\ovl[1]{\overline{#1}}
\newcommand{\sm}{\setminus}

\renewcommand{\ge}{\geqslant}
\renewcommand{\le}{\leqslant}

\newcommand{\Lefschetz}{\marginpar{Lefschetz}}
\renewcommand{\Lefschetz}{}

\theoremstyle{theorem}
\newtheorem{theorem}{Theorem}[section]
\newtheorem{lemma}[theorem]{Lemma}
\newtheorem{proposition}[theorem]{Proposition}
\newtheorem{corollary}[theorem]{Corollary}

\newtheorem{Theorem}{Theorem}

\theoremstyle{definition}
\newtheorem{definition}[theorem]{Definition}

\theoremstyle{remark}
\newtheorem*{remark}{\textsc{Remark}}

\newcounter{reminder}

\newtheoremstyle{claim}
  {}
  {}
  {\itshape}
  {0pt}
  {\scshape}
  {.}
  { }
  {\thmname{#1}\thmnumber{ #2}\thmnote{ (#3)}}

\theoremstyle{claim}
\newtheorem{claim}{Claim}

\numberwithin{equation}{section}

\title[Postcritically Fixed Newton Maps]{A Combinatorial Classification of Postcritically Fixed Newton Maps}
\author[K.\,Drach]{Kostiantyn Drach}
\author[Ya.\,Mikulich]{Yauhen Mikulich}
\author[J.\,R\"uckert]{Johannes R\"uckert}
\author[D.\,Schleicher]{Dierk Schleicher}
\address{Jacobs University Bremen, Research I, Campus Ring 1, 28759 Bremen, Germany}
\email{k.drach@jacobs-university.de}
\email{y.mikulich@gmail.com}
\email{jrueckert@jacobs-alumni.de}
\email{dierk@jacobs-university.de}

\thanks{Research was partially supported by the ERC advanced grant ``HOLOGRAM''. The second named author was supported by the Deutsche Forschungsgemeinschaft SCHL 670/2-1. The third author was partly supported by a Doktorandenstipendium of the German Academic Exchange Service (DAAD)}

\begin{document}
\begin{abstract}
We give a combinatorial classification for the class of postcritically fixed Newton maps of polynomials as dynamical systems. This lays the foundation for classification results of more general classes of Newton maps.

A fundamental ingredient is the proof that for every Newton map (postcritically finite or not) every connected component of the basin of an attracting fixed point can be connected to $\infty$ through a finite chain of such components.
\end{abstract}
\maketitle

\section{Introduction}

One of the most important open problems in rational dynamics is understanding the structure of the space of rational functions of a fixed degree $d\geqslant 2$. This problem is today wide open.

Newton maps have long been known as useful tools for numerical root finding at least when approximate roots are known, but has recently turned out to be extremely efficient also for finding all roots of complex polynomials of large degrees of several million \cite{NewtonEfficient, NewtonRobin} and even more than a billion \cite{NewtonRobinMarvin}. Moreover, Newton maps of polynomials form an interesting class of rational maps that is more accessible for studying than the full space of rational maps. Hence, a substantial intermediate goal in the classification of \emph{all} rational maps is to gain an understanding of the space of Newton maps; in fact, to this day, these maps form the largest family of rational maps beyond polynomials for which we have such an understanding (of the postcritically finite maps); this is a combination of our present paper and the subsequent work in \cite{LMS1,LMS2}. Newton maps form a fundamental ingredient in a description of general rational maps via a natural decomposition \cite{DecompositionDimaMisha}.

In this paper, we present a theorem that structures the dynamical plane of all Newton maps (postcritically finite or not), and then use this result to construct a graph that classifies those Newton maps whose critical orbits all terminate at fixed points. Newton maps of degree $1$ and $2$ are trivial, and we exclude these cases from our investigation. Let us make precise what we mean by a Newton map.

\begin{definition}[Newton map]
\label{Def_NewtonMap}
A rational map $f\colon\Cc\to\Cc$ of degree $d\ge 2$ is called a \emph{Newton map} if there is a polynomial $p\colon\C\to\C$ so that $f$ is the rational map arising in the Newton-Raphson root-finding method of $p$: that is,  $f(z)=z-p(z)/p'(z)$.
\end{definition}

For simplicity, we say that $f$ is the Newton map of $p$. The fixed points of $f$ in $\C$ are exactly the distinct roots of $p$: every such fixed point is attracting with multiplier $(m-1)/m$, where $m\ge 1$ is the multiplicity as a root of $p$. In particular, simple roots are superattracting. Every Newton map has another fixed point at $\infty$; it is repelling with multiplier $d'/(d'-1)$, where $d'$ is the degree of $p$.

A polynomial $p$ and its Newton map have the same degree if and only if all roots of $p$ are simple; in general, the degree of the Newton map equals the number of distinct roots of $p$.

Recall that our focus in this paper is on maps of degree $d\ge 3$: the case $d=2$ is trivial and excluded without explicit mention.

The rational maps that arise as Newton maps can be described explicitly as follows:
\begin{proposition}[Head's theorem]
\label{Prop_Head}
A rational map $f$ of degree $d \geqslant 2$ is a Newton map if and only if $\infty$ is a repelling fixed point of $f$ and for each fixed point $\xi\in\C$, there exists an integer $m\geqslant 1$ such that $f'(\xi)=(m-1)/m$.
\qed
\end{proposition}

This result is folklore; a proof can be found in \cite[Proposition~2.8]{RS}. The special case of postcritically finite Newton maps (for which all $m=1$) goes back to \cite[Proposition 2.1.2]{Head}. In general, for a root of $p$ of multiplicity $m$, the multiplier of the corresponding attracting fixed point of $f$ equals $(m-1)/m$, so that $f$ can be postcritically finite only if all roots of $p$ are simple.

\begin{remark}
There are also transcendental entire functions $h$ for which the  associated Newton maps are rational: these maps have the form $h=p e^q$ with polynomials $p$ and $q$. These too can be described (see again \cite[Proposition~2.8]{RS} and especially the recent PhD thesis of Khudoyor Mamayusupov \cite{KhudoyorThesis}); their dynamics is remarkably similar to the dynamics of Newton maps of polynomials as considered here; for details, see \cite{KhudoyorThesis,KhudoyorFundamenta,KhudoyorClassification}. 
\end{remark}

\begin{definition}[Immediate basin]
\label{Def_ImmediateBasin}
Let $f$ be a Newton map and $\xi\in\C$ a fixed point of $f$. Let $B_{\xi}=\{z\in\C\,:\, \lim_{n\to\infty}f^{\circ n}(z)=\xi\}$ be the \emph{basin (of attraction)} of $\xi$.
The connected component of $B_\xi$ containing $\xi$ is called the \emph{immediate basin} of $\xi$ and denoted $U_\xi$.
\end{definition}
Clearly, $B_\xi$ is open. By a theorem of Przytycki \cite{Przytycki}, $U_{\xi}$ is simply connected and $\infty\in\partial U_\xi$ is an accessible boundary point; in fact, a result of Shishikura \cite{Shishikura} implies that every component of the Fatou set is simply connected. 

Our first result is the following; it applies to \emph{all} (polynomial) Newton maps, not just postcritically fixed or finite ones.
\begin{Theorem}[Preimages connected]
\label{Thm_PreimagesConnected}
Let $f:\Cc\to\Cc$ be an arbitrary Newton map. Denote its fixed points in $\C$ by $\xi_1,\dots,\xi_d$. Then every component of every $B_{\xi_i}$ can be connected to $\infty$ by a curve through the closures of finitely many components of $\bigcup_{i=1}^d B_{\xi_i}$ so that only finitely many boundary points of $B_{\xi_i}$ are on the curve, and these are necessarily iterated preimages of $\infty$. 

More precisely, if $V_0$ is any component of any $B_{\xi_i}$, then there exist finitely many components $V_1\dots V_k$ of $\bigcup_{i=1}^d {B_{\xi_i}}$ and 
there exists a curve $\gamma:[0,1]\to\Cc$ such that $\gamma(0)\in V_0$, $\gamma(1)=\infty$ and for every $t\in[0,1]$, there exists $m\in\{0,1,\dots, k\}$ such that $\gamma(t)\in \ol{V}_m$, and so that $\gamma(t)\in\bigcup_{m=0}^k V_m$ for all $t$, except for finitely many $t$ for which $\gamma(t)$ are poles or prepoles. 
\end{Theorem}

Theorem \ref{Thm_PreimagesConnected} allows us to describe how the components of the basins are connected to each other. Thus, it is a basis for a combinatorial classification of those Newton maps where all critical points eventually land in immediate basins or on their common boundary point $\infty$: Theorems \ref{Thm_NewtonGraph} and \ref{Thm_Realization} (given below) show that the combinatorics of these connections suffice to describe such Newton maps uniquely (subject to the condition that all critical orbits are finite, which can be removed by a standard surgery).

\looseness-1
We call a Newton map \emph{postcritically fixed} if all its critical points land on fixed points after finitely many iterations (recall that fixed points are the roots of $p$ as well as $\infty$); this is much more restrictive than \emph{postcritically finite} maps (on which we comment below). If $f$ is a postcritically fixed Newton map, Theorem \ref{Thm_PreimagesConnected} makes it possible to structure the entire Fatou set because each Fatou component is in the basin of some superattracting fixed point.

Our combinatorial description starts with the channel diagram $\Delta$ of $f$: it is defined as the union of the accesses from the finite fixed points of $f$ to $\infty$ (see Section \ref{Sec_Models}). Denote by $\Delta_n$ the connected component of $f^{-n}(\Delta)$ containing $\infty$. It is a consequence of Theorem~\ref{Thm_PreimagesConnected} that for sufficiently large $n$, the connected graph $\Delta_n$ contains all critical points that are prepoles or are contained in the basins, as well as their forward orbits ---  similar to the Hubbard tree of a postcritically finite polynomial; in the postcritically \emph{fixed} case, these are all critical points. If $n$ is minimal so that $\Delta_n$ contains all critical points, then $\Delta_{n+1}$ will be the \emph{Newton graph associated to $f$}.

We introduce the notion of an \emph{abstract Newton graph}, which is a pair $(\Gamma,g)$ of a map $g$ acting on a finite graph $\Gamma$ that satisfies certain conditions; see Definition \ref{Def_NewtonGraph}. In particular, the conditions on $(\Gamma,g)$ allow $g$ to be extended to a branched covering $\overline{g}$ of the whole sphere $\S^2$.

We show that for every {abstract Newton graph} $(\Gamma,g)$ there exists a postcritically fixed Newton map for which $(\Gamma,g)$ is the associated Newton graph; we say that $f$ \emph{realizes} the abstract Newton graph $(\Gamma,g)$. In this case, the map $f$ is unique up to affine conjugacy. 

The assignments of a Newton map to an abstract Newton graph and vice versa are injective and inverse to each other, so a combinatorial classification of postcritically fixed Newton maps is provided in terms of abstract Newton graphs. Thus, our main results are the following (see Sections \ref{Sec_Graph} and \ref{Sec_Thurston} for the precise definitions).

\begin{Theorem}[Newton map generates Newton graph]
\label{Thm_NewtonGraph} 
For every postcritically fixed Newton map $f$ there exists a unique $N\in\N$ such that $(\Delta_N,f)$ is an abstract Newton graph.

If $f_1$ and $f_2$ are Newton maps with channel diagrams $\Delta_1$ and $\Delta_2$ such that $(\Delta_{1,N},f_1)$ and $(\Delta_{2,N},f_2)$ are equivalent as abstract Newton graphs, then $f_1$ and $f_2$ are affinely conjugate, and conversely.
\end{Theorem}

\begin{Theorem}[Newton graph generates Newton map]
\label{Thm_Realization}
Every abstract Newton graph is realized by a postcritically fixed Newton map, unique up to affine conjugacy, in the following sense: 
if $(\Gamma,g)$ is an abstract Newton graph then there exists a postcritically fixed Newton map $f$ for which the associated Newton graph is (equivalent to) $(\Gamma,g)$. 

Moreover, if $f$ realizes two abstract Newton graphs $(\Gamma_1,g_1)$ and $(\Gamma_2,g_2)$, then the two abstract Newton graphs are equivalent.
\end{Theorem}

Our construction of an abstract Newton graph can be done for all postcritically finite Newton maps, but will in general not contain the orbits of \emph{all} critical points, and thus not describe the combinatorics of the entire Fatou set: for instance, when there are periodic critical points, their associated Fatou components (and thus the dynamics of the critical orbits contained therein) are not described by the Newton graph $\Delta_n$. 

The classification of postcritically \emph{finite} Newton maps in \cite{LMS1,LMS2} is based on ``extended Newton graphs'': these are our Newton graphs extended by finitely many Hubbard trees that describe the dynamics of those critical points that do not land on fixed points; these Hubbard trees need to be connected to the Newton graph by curves that we call Newton rays (and this connection is the most difficult part).

Thus, our results are a first step, and provide necessary tools, towards a combinatorial classification of all postcritically finite Newton maps, and hence of all hyperbolic components in the space of Newton maps. They may also be a basis for transporting the powerful concept of Yoccoz puzzles, which has been used to prove local connectivity of the Julia set for many classes of polynomials, to the setting of Newton maps beyond the cubic case (Roesch has successfully applied Yoccoz puzzles to cubic Newton maps \cite{Roesch}). Moreover, they provide a fundamental ingredient in a natural decomposition of arbitrary rational maps into \emph{Newton-like} maps and maps of Sierpi\'nski type \cite{DecompositionDimaMisha}.

A number of people have studied Newton maps and used combinatorial models to structure the parameter spaces of some Newton maps. Janet Head \cite{Head} introduced the \emph{Newton tree} to characterize postcritically finite cubic Newton maps. Tan Lei \cite{TanLei}  built upon this work and gave a classification of postcritically finite cubic Newton maps in terms of matings and captures.  Jiaqi Luo \cite{Luo} extended some of these results to \emph{unicritical} Newton maps, i.e., Newton maps of arbitrary degree with only one \emph{free} (non-fixed) critical value.
The present work can be seen as an extension of these results beyond the setting of a single free critical value, and thus beyond the setting where all maps come from complex one-dimensional parameter spaces. The main differences to the case of only one free critical points are that the channel diagram can have more than one branch point in $\C$ and that in the presence of more than one non-fixed critical value, the iterated preimages of the channel diagram may be disconnected (to deal with this problem is our main concern in Sections~\ref{Sec_Models} and \ref{Sec_Proof}).

\smallskip
This article is structured as follows. In Section \ref{Sec_Models}, we introduce the concept of a channel diagram for Newton maps and discuss some of its properties, notably with the help of a useful fixed point theorem. We use the channel diagram and its preimages to prove Theorem \ref{Thm_PreimagesConnected} in Section \ref{Sec_Proof}. In Section \ref{Sec_Graph}, we introduce abstract Newton graphs and prove Theorem \ref{Thm_NewtonGraph}. Theorem \ref{Thm_Realization} is proved in Section \ref{Sec_Thurston}, following a review of fundamental aspects of Thurston theory. We also give an introduction to the combinatorics of \emph{arc systems} and state a result by Kevin Pilgrim and Tan Lei that restricts the possibilities of how arc systems and Thurston obstructions can intersect.

\subsection{Notation}

Let $f$ be a Newton map. A point $z \in \C$ is called a \emph{pole} if $f(z)=\infty$ and a \emph{prepole} if $f^{\circ k}(z)=\infty$ for some $k> 1$. 
If $g:\Cc \to \Cc$ is a branched covering map, we call a point $z \in \Cc$ a \emph{critical point of $g$} if $g$ is not injective in any neighborhood of $z$. For the Newton map $f$, this is equivalent to saying that $z\in\C$ and $f'(z)=0$ because $\infty$ is never a critical point of $f$. It follows from the Riemann-Hurwitz formula \cite[Theorem 7.2]{Milnor} that a degree $d$ branched covering map of $\Cc$ has exactly $2d-2$ critical points, counting multiplicities. A \emph{closed topological disk} will be a subset of $\hat \C$ homeomorphic to $\ovl \disk$ (with Jordan boundary), and we set $D_r(a):=\{z\in\C\colon |z-a|<r\}$.

\begin{definition}[Postcritically fixed]
\label{Def_PostCritFixed}
Let $g:\Cc \to \Cc$ be a branched covering map of degree $d \geqslant 2$ with (not necessarily distinct) critical points $c_1,\dots,c_{2d-2}$. 
Then $g$ is called \emph{postcritically finite} if the set
\[
P_g := \bigcup_{i} \bigcup_{n \geqslant 1} g^{\circ n}(c_i)
\]
is finite.
We say that $g$ is \emph{postcritically fixed} if there exists $N\in\N$ such that for each $i\in\{1,\dots,2d-2\}$ the point $g^{\circ N}(c_i)$ is a fixed point of $g$.
\end{definition}

\begin{definition}[Access to $\infty$]
\label{Def_Access}
Let $U_{\xi}$ be the immediate basin of the attracting fixed point $\xi \in \C$. Consider an injective curve $\Gamma:[0,1]\to U_{\xi}\cup\{\infty\}$ with $\Gamma(0) =\xi$ and $\Gamma(1)=\infty$. Its homotopy class within $U_{\xi}\cup\{\infty\}$, fixing endpoints, defines an \emph{access to $\infty$} for $U_{\xi}$; in other words, a curve $\Gamma'$ with the same properties lies in the same access as $\Gamma$ if the two curves are homotopic in $U_{\xi}\cup\{\infty\}$ with the endpoints fixed.

\begin{remark}
In topologically simple cases, a simpler definition suffices. For instance for a finite graph $\Gamma$ embedded in the sphere, an access to a vertex $x\in \Gamma$ can simply be defined in terms of a (sufficiently small) disk $D$ around $x$: an access is then represented by a component of $D\sm\Gamma$ that contains $x$ on the boundary. We use this point of view in later sections. 
\end{remark}

\end{definition}
\begin{proposition}[Accesses; \protect{\cite[Prop.~6]{HSS}}]
\label{Prop_Access}
Let $f$ be a Newton map of degree $d\geqslant 2$ and $U_\xi$ an immediate basin for $f$. Then there exists $k\in \{1,\dots,d-1\}$ such that $U_\xi$ contains $k$ critical points of $f$ (counting multiplicities), $f|_{U_\xi}$ is a branched covering map of degree $k+1$, and $U_\xi$ has exactly $k$ accesses to $\infty$.
\qed
\end{proposition}

\medskip
\emph{Acknowledgements}.
We thank Tan Lei and the dynamics group in Bremen, especially Russell Lodge, for their comments that helped to
improve this paper. Moreover, we appreciate very useful suggestions by the referees on earlier versions of this paper.

We gratefully acknowledge support by the Deutsche Forschungsgemeinschaft DFG (YM) and the European Research Council ERC (KD, DS).

\section{The Channel Diagram and the Structure of the Immediate Basins and their Complements}
\label{Sec_Models}

In the following, by a \emph{graph} we mean a connected topological space $\Gamma$ homeomorphic to the quotient of a finite disjoint union of closed arcs by an equivalence relation on the set of their endpoints; our graphs are finite. The arcs are called \emph{edges} of the graph, an equivalence class of endpoints a \emph{vertex}.

Further we will understand the closure and boundary operators performed with respect to the topology of the Riemann sphere $\Cc$ unless otherwise stated. Also, we will say that a set $X\subset\Cc$ is \emph{bounded} if $\infty\not\in\ol{X}$.

We say that a Newton map $f$ of degree $d$ is \emph{attracting-critically-finite } if it has the following property:
\begin{equation} 
\tag{$\star$} 
\label{Prop_PCF}
\left\{
\begin{aligned}
& \text{if } c \text{ is a critical point of } f \text{ with } c\in B_{\xi_j} \text{ for some } \\
& j\in\{1,\dots,d\}, \text{ then } c \text{ has finite orbit;}
 \end{aligned}
 \right.
\end{equation}
in other words, all critical points in the basins of the roots have finite orbits, or equivalently, all attracting fixed points are superattracting and all critical orbits in their basins eventually terminate at the fixed points. 

The following observation is well known, and its proof is standard and omitted.
\begin{lemma}[Only critical point]
\label{Lem_OnlyCritical}

Let $f$ be a Newton map that is attracting-critically-finite and let
$\xi\in\C$ be a fixed point of $f$ with immediate basin $U_{\xi}$.
Then $\xi$ is the only critical point in $U_{\xi}$. \qed
\end{lemma}

By a standard construction of quasiconformal surgery, every Newton map can be turned into an attracting-critically-finite Newton map of the same degree so that the restriction of both maps to the Julia set is topologically conjugate, and it will preserve the statement in Theorem~\ref{Thm_PreimagesConnected}. For the most part in Sections~\ref{Sec_Models} and \ref{Sec_Proof}, we will thus be able to work with attracting-critically-finite maps, and relate this to the general case at the end of Section~\ref{Sec_Proof}. Postcritically fixed (and postcritically finite) maps as discussed in Sections~\ref{Sec_Graph} and \ref{Sec_Thurston} satisfy this condition anyway. 
 
Focusing now on attracting-critically-finite Newton maps, for every immediate basin $U_\xi$ there is an inverse Riemann map $\phi_\xi\colon\disk\to U_\xi$ with $\phi_\xi(0)=\xi$ and with the
property that $f(\phi_\xi(z)) = \phi_\xi(z^{k_\xi})$ for each $z\in
\disk$, where $k_\xi-1\geqslant 1$ is the multiplicity of $\xi$ as a
critical point of $f$ \cite[Theorems 9.1 and 9.3]{Milnor}; this map $\phi_\xi$ is also known as B\"ottcher map. An \emph{internal ray} is the image under $\phi_\xi$ of a radial line in $\disk$, so that every internal ray in $U_\xi$ maps under $f$ to another internal ray. In particular, there are $k_\xi-1$ fixed internal rays. These give $k_\xi-1$ invariant curves $\Gamma_{\xi}^1,\dots,\Gamma_{\xi}^{k_\xi-1}$ in $U_{\xi}$ that connect $\xi$ to $\infty$; they are disjoint and have disjoint closured except for at $\xi$ and $\infty$. These invariant curves represent different accesses to $\infty$ of $U_\xi$, and every such access is represented by one of them (Proposition \ref{Prop_Access}).  Hence if $\xi_1,\dots,\xi_d\in\C$ are the attracting fixed points of $f$, then the union
\begin{equation}
    \Delta := \bigcup_{i=1}^d\bigcup_{j=1}^{k_{\xi_i}-1} \ol{\Gamma_{\xi_i}^j}
\label{Eq:DefChannelDiagram}
\end{equation}
of these invariant curves over all immediate basins forms a
connected and $f$-invariant graph in $\Cc$ with vertices at the
$\xi_i$ and at $\infty$.  We call $\Delta$ the \emph{channel diagram}
of $f$. 

The channel diagram records the mutual locations of the
immediate basins of $f$ and provides a first-level combinatorial
structure to the dynamical plane. Figure \ref{Fig_Degree6} shows a
Newton map and its channel diagram. In Definition~\ref{Def_ChannelDiagram} we will give an axiomatic characterization of channel diagrams.

\begin{figure}[hbt]
\begin{center}
\setlength{\unitlength}{1cm}
\begin{picture}(10,10)
{
\includegraphics[width=0.8\textwidth]{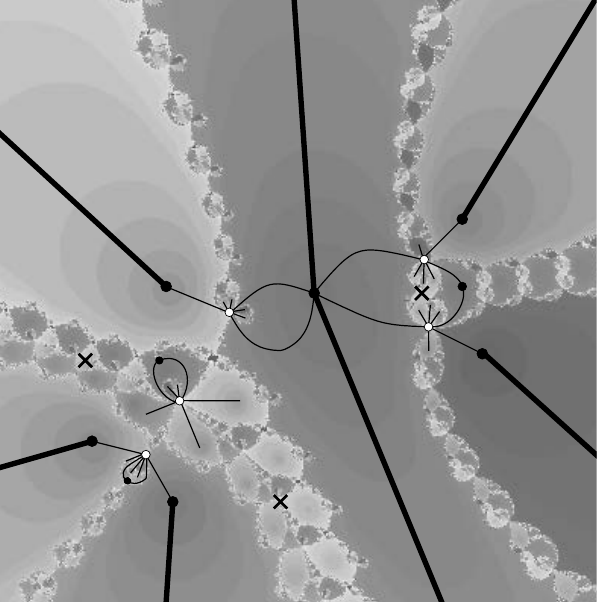}
}
\end{picture}
\caption{\label{Fig_Degree6} A Newton map of degree 6, superimposed with its channel diagram: the thick black lines represent the fixed rays of the immediate basins, the large black dots correspond to the fixed points (the fixed point at $\infty$ is not visible).
The thick and thin black lines together show the first preimage of the channel diagram: white circles represent poles, small black dots indicate preimages of the central fixed point, and a cross is a ``free'' critical point. The Newton map has a pole outside $\Delta_1$ (visible in the lower left), and the right boundary component of the central immediate basin contains two poles.}
\end{center}
\end{figure}

The following theorem is the main result of this section; it shows a relation between poles and fixed
points outside immediate basins. It considerably sharpens
\cite[Corollary 5.2]{RS}, which states that for an immediate basin
$U_{\xi}$ of a Newton map, every component of $\C\setminus U_{\xi}$
contains at least one fixed point (i.e.\ a root and hence an immediate basin).

\begin{theorem}[Fixed points and poles]
\label{Thm_FixedPoles}
Let $f$ be a Newton map and $U_\xi$ an immediate basin. If $W$ is a component of $\C\setminus U_\xi$, then the number of fixed points in $W$ equals the number of poles in $W$, counting multiplicities.
\end{theorem}

Let us first give a rough outline of the proof idea before getting involved in somewhat technical constructions. Suppose for simplicity that $\ovl W=W\cup\{\infty\}$ is a closed topological disk so that $\partial W=:\Gamma$ is a circle and $f\colon\Gamma\to\Gamma$ is a covering map of degree $k\ge 2$. First consider the special case that $f(\ovl W)=\ovl W$; then $f$ is a branched cover of degree $k$ that has $k$ fixed points on $\ovl W$ and every point, in particular $\infty$, has $k$ preimages, so there are $k-1$ poles in $W$ (all on $\partial W\sm\{\infty\}$); since one of the fixed points is $\infty$, the set $W$ contains $k-1$ fixed points and the claim holds. 

Usually, however, $f(\ovl W)=\hat\C$, so $\ovl W\sm f^{-1}(\Gamma)$ consists of finitely many components, each of which either maps to $\mathring W$ or to $\C\sm W$. If $W_0$ is such a component that maps to $\mathring W$, then it contributes fixed points according to its mapping degree, and it has the same number of poles on the boundary. A component that maps to $\C\sm W$ does not have fixed points, and poles are on the boundary where they are already counted for the components of the first kind. 

While the overall idea of this argument makes sense, its details do not: we do not know whether $\partial U_\xi$ is a curve, and in general it will not be a Jordan curve (if $U_\xi$ has two or more accesses, then every pole or prepole on $\partial U_\xi$ must locally disconnect $U_\xi$). 

To produce the precise proof we will construct a curve $\Gamma\subset U_\xi\cup\{\infty\}$ ``near'' $\partial U_\xi$, and for the counting of fixed points and poles we use the following  useful observation that is a special case of \cite[Theorem 4.8]{RS}. 

\begin{lemma}[Fixed points]
\label{Lem_Lefschetz}
Let $f$ be a Newton map and let $D\subset\Cc$ be a closed topological disk such that $\gamma:=f(\partial D)$ is a simple closed curve with the property that $\gamma\cap\mathring{D}=\emptyset$. Let $V$ be the unique component of\/ $\Cc\setminus\gamma$ that contains $\mathring{D}$ and let $(\gamma'_i)_{i\in I}$ be the collection of boundary components of $f^{-1}(V)\cap D$. 
Suppose in addition that any fixed point $p$ of\/ $f$ on $\partial D$ is repelling and has a neighborhood $U$ so that $f(\partial D\cap U)\subset\partial D$. 

Then the number of fixed points of $f$ in $D$ equals
\[
    \sum_{i\in I} \left|\deg(f|_{\gamma'_i}:\gamma'_i\to\gamma)\right|.
\]
In particular, if $f^{-1}(V)\cap D\neq\emptyset$, then $D$ contains a fixed point.
\qed
\end{lemma}
\begin{remark}
The general statement of \cite[Theorem 4.8]{RS} allows the existence of parabolic fixed points. Since these do not exist for $f$, we do not need to take multiplicities of fixed points into account (in the notation of \cite{RS}, all fixed point indices $\iota(\xi,f)=1$). 
\end{remark}

\begin{proof}[Proof of Theorem~\ref{Thm_FixedPoles}]
Let $d$ be the degree of $f$. If $U_\xi$ does not separate the plane, i.e., if it has only one access to $\infty$, then the claim follows trivially: $W$ contains all $d-1$ finite poles and the $d-1$ finite fixed points of $f$ other than $\xi$. So suppose in the following that there is an inverse Riemann map $\phi:(\disk,0)\to (U_\xi,\xi)$ with $f(\phi(z))=\phi(g(z))$
for some Blaschke product $g:\disk\to\disk$ of degree $k:=\deg(f|_{U_{\xi}}:U_{\xi}\to U_{\xi})\geqslant 3$.

\begin{figure}[hbt]
\begin{center}
\includegraphics[trim=0 10 0 10]{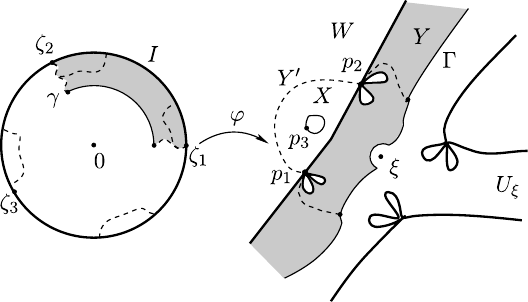}
\caption{\label{Fig_LefschetzCurve} In the proof of Theorem \ref{Thm_FixedPoles}, the construction of $\gamma\subset\disk$ is shown on the left for $k=4$. The dashed curves are the components of
$g^{-1}(\gamma)$. The picture on the right shows the curve $\Gamma\subset U_{\xi}$ for the case $p_1\neq p_2$. The dashed curve indicates where $\Gamma'_1$ differs from $\Gamma$.
}
\end{center}
\end{figure}

We may extend $g$ by Schwarz reflection to a rational function $\hat{g}:\Cc\to\Cc$ of degree $k$ whose Julia set equals $\S^1$ and that has $k-1$ fixed points $\zeta_1,\dots,\zeta_{k-1}\in \S^1$. These fixed points correspond to the accesses to $\infty$ of $U_{\xi}$. Since $\hat{g}$ fixes $\disk$, the $\zeta_i$ have real positive multipliers and since $0$ and $\infty$ attract all of $\disk$ and of
$\Cc\setminus\ol{\disk}$, respectively, none of the $\zeta_i$ can be attracting or parabolic, so they are pairwise distinct and repelling. For each $\zeta_i$ choose a linearizing neighborhood and choose $\rho\in(0,1)$ large enough so that all critical values of $\hat{g}$ in $\disk$ are in $D_\rho(0)$ and so that the linearizing neighborhoods of all $\zeta_i$ intersect the circle at radius $\rho$. 

Now we construct a curve $\gamma\subset\disk$ as follows (see Figure \ref{Fig_LefschetzCurve}): 
there is a unique pair of adjacent fixed points $\zeta_j,\zeta_{j+1}\in\S^1$ so that the corresponding accesses to $\infty$ in $U_\xi$ separate $W$ from all other components of $\C\setminus U_\xi$. Let $\gamma_j$ be the unique curve in $\disk$ that is a straight line segment in linearizing coordinates of $\zeta_j$ and that connects $\zeta_j$ to the circle at radius $\rho$; and there is a similar curve $\gamma_{j+1}$ near $\zeta_{j+1}$. Their endpoints cut the circle at radius $\rho$ into two arcs. Of those arcs, let $\gamma'$ be the one for which $\gamma:=\gamma_j\cup \ol{\gamma'}\cup \gamma_{j+1}$ has the property that $\phi(\gamma)\subset U_{\xi}$ separates $W$ from $\xi $. Let $\Gamma:=\phi(\gamma)\cup\{\infty\}$. Then $\Gamma$ is a simple closed curve in ${U}_{\xi}\cup\{\infty\}$ and contains no critical values, except possibly $\infty$, and it is uniquely specified by the construction.

Let $Y$ be the closure of the component of $\Cc\setminus \Gamma$ that contains $W$. Then $Y\supset W$, so $Y$ is a Jordan domain that replaces $W$ in the simple argument given before the proof. Since $Y\sm W\subset U_\xi\sm\{\xi\}$, the number of poles in $W$ and in $Y$ coincide, and also the number of fixed points.

Let us first suppose that $\infty$ is not a critical value. Then every component $\Gamma'_i$ of $f^{-1}(\Gamma)$ is a simple closed curve and $\deg(f|_{\Gamma'_i}:\Gamma'_i\to\Gamma)$ equals the number of poles on $\Gamma'_i$ (no need to count multiplicities). 

Let $\Gamma_1'$ be the component of $f^{-1}(\Gamma)$ containing $\infty$. We claim that $\Gamma'_1\cap U_{\xi}$ consists of two connected components, each of which is an injective curve that connects $\infty$ to a pole on $\partial U_{\xi}$; call these poles $p_1$ and $p_2$ (possibly $p_1=p_2$). Indeed, consider the situation in $\disk$-coordinates (see Figure~\ref{Fig_LefschetzCurve}, left). Let $I\subset\S^1$ be the arc between $\zeta_j$ and $\zeta_{j+1}$ that is separated from $0$ by $\gamma$; then $I$ and $\gamma$ are homotopic relative to the critical values of $\hat g$. Since $\mathring{I}$ contains no fixed points of $\hat{g}$, it follows that $\hat{g}(\ol{I})$ covers $\S^1\setminus \ol{I}$ exactly once and $\ol{I}$ itself exactly twice. Therefore, $\hat{g}^{-1}(\gamma)$ has exactly two connected components that intersect $\gamma$, as claimed.

Let $Y'$ be the closure of the component of $\Cc\setminus\Gamma'_1$ that intersects $W$ in an unbounded set.

If $p_1=p_2$, then $\Gamma'_1\subset {U}_{\xi}\cup\{p_1,\infty\}$, otherwise $\Gamma'_1\not\subset \ol{U}_{\xi}$ (this is the situation pictured in Figure \ref{Fig_LefschetzCurve}: $\Gamma'_1$ is a simple closed curve that connects $p_1$ to $p_2$ through a component of $f^{-1}(U_\xi)$ other than $U_\xi$ itself). If the set $X:=W\setminus Y'$ contains neither poles nor fixed points of $f$, then the numbers of fixed points and poles in $W$ and $Y'$ coincides. But since $Y'$ is a closed topological disk and $f(\partial Y')\cap \mathring{Y}'=\emptyset$, Lemma~\ref{Lem_Lefschetz} \Lefschetz
 shows that the number of fixed points in $Y'$ (including $\infty$) equals the number of poles in $Y'$ (again including $\infty$) because on every component of $f^{-1}(\Gamma)$ in $Y'$ the degree of $f$ equals the number of poles it contains. Excluding $\infty$ again, the claim follows.

We now show that, in fact, $X$ never contains a fixed point. First observe that $\Gamma$ and $\Gamma'_1$ coincide in a neighborhood of $\infty$, so these two curves together surround a bounded subset of $\C$ that contains $X$. Any fixed point in $X$ would have to be a root of $p$ and thus be surrounded by its immediate basin, which is unbounded and cannot intersect $\Gamma\cup\Gamma'_1$, a contradiction. 
Therefore, $W$ contains at least as many poles as fixed points. Since this is true for all components of $\C\setminus U_\xi$, and all of them combined contain all $d-1$ poles in $\C$ and all $d-1$ fixed points in $\C\setminus\{\xi\}$, the number of poles and fixed points in each component $W$ must be the same.

We still have to treat the case that $\infty$ is a critical value. In this case, we perturb $f$ slightly (among Newton maps) to avoid that situation. One problem is that $U_{\xi}$ and $W$ might move discontinuously, and $U_\xi$ might even acquire (but not lose) additional channels under small perturbations. However, we argued above that in any case, the numbers of fixed points and poles that we are interested in are the same in $W$ and in $Y$, and under small perturbations poles and fixed points of $f$ move continuously, and so do $\Gamma$ and hence $Y$ (even if $U_\xi$ suddenly acquires additional channels). Depending on the perturbation, the topology of $\Gamma'_1$ might change: it may no longer be a simple closed curve if it contains critical points, but it will be a finite graph that coincides with $\Gamma$ in a neighborhood of $\infty$. Any bounded subset of $\C$ that is bounded by parts of $\Gamma$ and $\Gamma'_1$ cannot contain fixed points before or after perturbation (by the same argument as above), and after perturbation the reasoning above applies as before. This shows that $Y$ still contains at least as many poles as fixed points, and since this is true for all components of $\C\sm U_\xi$, the conclusion follows in this case too.
\end{proof}

\begin{corollary}[Poles on boundary component]
\label{Cor_BoundaryContainsPole}
For every immediate basin $U_\xi$, every component of $\partial U_\xi\cap\C$  contains exactly one or two poles. 
\end{corollary}

\begin{proof}
The components of $\partial U_{\xi}\cap\C$ are separated by the accesses to $\infty$. In the conjugate dynamics $\hat{g}|_{\disk}$ the map $g$ is a Blaschke product of degree at least two, so $\hat{g}\colon \S^1\to \S^1$ is a covering map of degree $k\ge 2$. Fixed points of $\hat g$ correspond to accesses of $U_\xi$ to $\infty$, and between any two consecutive fixed points on $\S^1$ there must be pre-fixed points, and these correspond to poles of $f$. This proves the existence of at least one pole. (More precisely, between any two fixed points of $\hat g$ on $\S^1$ there are $k-1$ pre-fixed points, but they might and often do correspond to the same pole of $f$.) 

The proof that there are at most two poles is included in the proof of Theorem~\ref{Thm_FixedPoles}: if there are at least three poles, then the set $X$ in that proof must contain one of them, but we proved this is not the case.
\end{proof}

Figure \ref{Fig_Degree6} shows that a component of $\partial U_{\xi}\cap \C$ may indeed contain two poles.

Our next result is slightly easier to state in the attracting-critically-finite case, so we state it first in that case, followed by a more general version that implies the first. 

\begin{corollary}[Fixed points in complement, attracting-critically-finite]
\label{Cor_FixedComplement}
Let $f$ be an attracting-critically-finite Newton map and let $\Delta$ be the channel diagram of $f$. For a component $V$ of $\C\setminus\Delta$ let $p$ be the number of poles of $f$ in $V$, counting multiplicities. Then $\partial V \cap \C$ contains $p+1$ fixed points.
\end{corollary}

\begin{corollary}[Fixed points in complement]
\label{Cor_FixedComplementGeneral}
Let $f$ be a Newton map and let $T:= \bigcup_\xi U_\xi$. For a component $V$ of $\C\sm T$ let $p$ be the number of poles of $f$ in $V$, counting multiplicities. Then $\partial V\cap\C$ intersects the boundaries of exactly $p+1$ immediate basins. 
\end{corollary}

\begin{proof}
If $V$ is the only component of $\Cc\setminus T$ (so every root has a single access to $\infty$), the claim is obvious with $p=d-1$. The second case is that there is a unique immediate basin $U_\xi$ with at least two accesses so that $\partial U_\xi\cap\partial V\cap\C\neq\emptyset$. In this case, the claim follows directly from Theorem~\ref{Thm_FixedPoles}. Indeed, let in this case $V_1$ be the component of $\C\sm U_\xi$ containing $V$. Then $V_1$, like $V$, contains $p$ poles and by Theorem~\ref{Thm_FixedPoles} contains also $p$ fixed points, hence $p$ immediate basins. Also counting $U_\xi$ as well, the claim follows.

\begin{figure}[hbt]
\begin{center}
\setlength{\unitlength}{1cm}
\begin{picture}(6,2.8)
\put(0,0){\includegraphics[viewport=0 0 337
158,width=6cm]{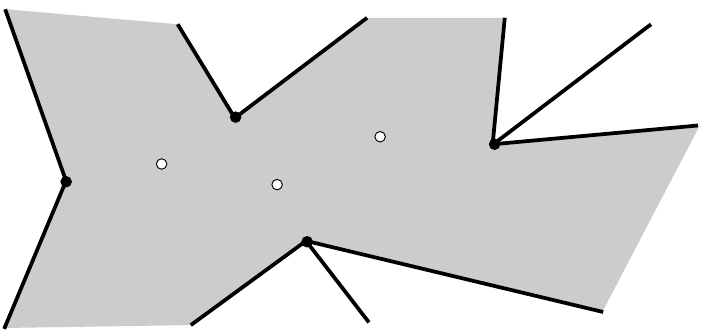}} \put(0.9,0.2){$V$}
\put(0.7,1){$\xi_1$} \put(2.8,0.9){$\xi_2$} \put(3.8,1.4){$\xi_3$}
\put(2,1.5){$\xi_4$} \put(1.9,2.3){$V_4^1$} \put(5.3,2){$V_3^1$}
\put(4.5,2.4){$V_3^2$} \put(3.4,0.1){$V_2^1$} \put(2.3,0.1){$V_2^2$}
\end{picture}
\caption{\label{Fig_FixedComplement} The case $p=3$ in Corollary \ref{Cor_FixedComplement}: the open set $V$ (shaded) is bounded by parts of $\Delta$ (represented by straight black lines). The black dots are fixed points, the white dots represent poles. The $V_j^i$ may well contain further structure of $\Delta$. For Corollary~\ref{Cor_FixedComplementGeneral}, the situation is similar, except that $T$ is a neighborhood of $\Delta$.}
\end{center}
\end{figure}

Now suppose that $U_{\xi_1}$, \dots, $U_{\xi_k}$ are the immediate basins that intersect $\partial V$ in $\C$ and that each have at least two accesses to $\infty$. As above, let $V_1$ be the component of $\C\sm U_{\xi_1}$ containing $V$. Let $m$ be the number of poles in $V_1$. As before, it follows that $V_1$ contains $m$ fixed points. Let $m':=m-p$. For $j=2,\dots,k$, denote by $V_j^1,\dots,V_j^{i_j}$ all components of $\C\sm U_{\xi_j}$ that do not contain $V$ (see Figure~\ref{Fig_FixedComplement}). By Theorem~\ref{Thm_FixedPoles}, each $V_j^{i}$ contains as many poles as fixed points, hence all $V_j^i$ combined contain the missing $m'$ poles and thus $m'$ fixed points. This implies that $V_1\setminus(\bigcup_{j=2}^k\bigcup_{\ell=1}^{i_j} V_j^{\ell})$ contains $m-m'=p$ fixed points, and their immediate basins intersect $\partial V$ in $\C$. Including $U_{\xi_1}$ in the count, we find $p+1$ immediate basins as claimed.
\end{proof}

\begin{corollary}[Existence of shared poles]
\label{Cor_SharedPoles}
Let $f$ be a Newton map and let $V$ be a component of $\C\sm \bigcup_\xi U_\xi$. Then $V$ contains at least one pole that is on the common boundary of two immediate basins.
\end{corollary}
\begin{proof}
If $p$ denotes the number of poles of $f$ in $V$, then by Corollary \ref{Cor_FixedComplementGeneral} $\partial V\cap \C$ intersects the boundaries of $p+1$ immediate basins. Each of these immediate basins has one boundary component (in $\C$) in $V$, and each of these boundary components contains at least one pole by Corollary~\ref{Cor_BoundaryContainsPole}. Since there are only $p$ poles available, at least one pole must be shared. 
\end{proof}
\begin{remark}
Note that every simple pole is on the boundary of at most two immediate basins because otherwise $f$ cannot preserve the cyclic order of the immediate basins near that pole. This was first observed by Janet Head \cite{Head}.
\end{remark}

\section{Proof of Theorem \ref{Thm_PreimagesConnected}}
\label{Sec_Proof}

The main work in this section consists in proving Theorem~\ref{Thm_PreimagesConnected} in the attracting-critically-finite case; we then deduce the general case using standard surgery techniques. 

Let us consider an attracting-critically-finite Newton map $f$ with fixed points $\xi_1,\dots, \xi_d\in\C $ and with channel diagram $\Delta$. Recall that $\Delta$ consists of invariant rays within $U_{\xi_i}$ that connect $\xi_i$ to $\infty$. Denote by $\Delta_n$ the connected component of $f^{-n}(\Delta)$ that contains $\Delta$ (with this convention, $\Delta=\Delta_0$). Every edge of $\Delta_n$ is then an internal ray of a component of some basin $B_{\xi_i}$, while every vertex is either $\xi_i$, or $\infty$, or an iterated preimage of these. By construction, $\Delta_{n} \subset \Delta_{n+1}$ for all $n \geqslant 0$.

The crucial ingredient for the whole proof of Theorem~\ref{Thm_PreimagesConnected} is to show that all $\Delta_n$ for large enough $n$ contain all the poles of $f$. To do that, we now explore consequences of the assumption that some pole is not in $\Delta_n$ for all $n$. Figure \ref{Fig_Degree6} gives an example where a pole is not in $\Delta_1$ (but it is in $\Delta_2$). 

\begin{proposition}[Pole is either in $\Delta_n$ or in infinite unbounded nest]
\label{Prop_PoleNotConnected}
For every pole $p_*$ of $f$ either $p_*\in\Delta_n$ for some $n$, or there exist a component $V_0$ of $\C\setminus \Delta$ and components $V_n$ of $f^{-n}(V_0)$, for all $n \geqslant 1$, with the following properties:
\begin{itemize}
\item
$V_0\supset V_1\supset \dots \supset V_n \supset V_{n+1}$ for all $n$;
\item
every $V_{n+1}$ is a component of $f^{-1}(V_n)$, that is a component of \goodbreak $\Cc \sm f^{-(n+1)}(\Delta) = f^{-(n+1)}(V_0)$;
\item
all $V_n$ are unbounded;
\item
all $V_n$ with $n\ge 1$ are multiply connected and the pole $p_*$ lies in a bounded component of $\partial V_n$.
\end{itemize}
\end{proposition}

The proof of Proposition~\ref{Prop_PoleNotConnected} requires a few preparations; the first step  is the following lemma. 

\begin{lemma}[Preimage inside]
\label{Lem_Inside1}
If there is a pole, say $p_*$, that is not in $\Delta_1$, then there exists a component $V_0$ of $\C\sm \Delta$ and a component $V_1$ of $f^{-1}(V_0)$ such that $V_1$ is multiply connected, $V_1\subset V_0$, and $p_*$ lies in a bounded component of $\partial V_1$.
\end{lemma}

\begin{proof}
There exist one or several components of $f^{-1}(\C\sm\Delta)=\C\sm f^{-1}(\Delta)$ that separate $p_*$ from $\infty$, so they are multiply connected and nested (in the sense that they are disjoint and each one separates the previous one from $p_*$). 
Let $V_1$ be one such component that has $p_*$ on its boundary; it exists since $\infty \in \partial f(V_1)$. Then there is a component $V_0$ of $\C\sm\Delta$ so that $V_1$ is a component of $f^{-1}(V_0)$. Since $\Delta$ is connected, $V_0$ is simply connected.

It remains to prove that $V_1\subset V_0$. Since $\Delta\subset f^{-1}(\Delta)$, we either have $V_1\subset V_0$ or $V_1\cap V_0=\emptyset$.

Let $\gamma\subset V_0$ be a simple closed curve near $\partial V_0$
that surrounds all critical values within $V_0$ (the domain $V_0$ must
contain critical values because it is the image of the multiply
connected domain $V_1$); obviously $\gamma$ does not contain a fixed point of $f$. Then $f^{-1}(\gamma)\cap V_1$ consists of
several simple closed curves, one near each boundary component of $V_1$. Let $\gamma'$ be the outermost of these curves (exactly one boundary component of $V_1$ either contains $\infty$ or separates $\infty$ from $V_1$, and $\gamma'$ is the component of $f^{-1}(\gamma)\cap V_1$ near this boundary component). 

Let $D$ be the bounded component of
$\Cc\setminus \gamma'$ and $V$ the  component of $\Cc \setminus \gamma$ containing $D$; thus $\gamma=\partial V$. If we have $V_1\cap V_0=\emptyset$, then one component of $\Cc\sm\gamma$ is contained in $V_0$, while the other component is $V$ and contains $V_1$ and hence $\gamma'$, as well as at least one more component of $f^{-1}(\gamma)\subset V_1$, say $\gamma''$. 
We thus have $\gamma''\subset D$, so Lemma~\ref{Lem_Lefschetz}{\Lefschetz} implies that $\ol{D}$ contains a fixed point of $f$ (in this case, this can be seen more directly since $f(\gamma'')=\gamma$ and the image under $f$ of the domain surrounded by $\gamma''$ must cover $V$ and hence itself). 
This is a contradiction because
$\ol{D}$ is disjoint from $\Delta$ and all fixed points are
contained in $\Delta$; thus $V_1\subset V_0$ as claimed.
\end{proof}

Suppose $V$ is a component of $\Cc \sm f^{-n}(\Delta)$ for some $n \geqslant 1$. Whether or not $V$ is simply connected, there exists a unique component of $\partial V$ in $\Cc$ that either contains $\infty$ or separates $V$ from $\infty$. We call this component the \emph{outer boundary} of $V$ and denote it by $Z_V$. The outer boundary is a finite graph, and if $V$ is unbounded, then $\ovl V \cap \Delta_n = Z_V$. Denote by $D_V$ the unique component of $\Cc \sm Z_V$ intersecting $V$. By definition, $D_V$ is a topological disk such that $Z_{D_V} = \partial D_V = Z_V$ and $D_V\supset V$.
We will later use the outer boundary $Z_V$ of other domains, as well as the filled-in disk $D_V$, in an analogous way and with analogous notation.

By construction, $Z_V$ is a finite graph; its vertices are $\infty$, poles, prepoles, fixed points, or pre-fixed points. 
In most cases $Z_V$ is not a Jordan curve: it is usually a graph, and some of its vertices may have three or more edges attached, or only one; and it may have edges that are on $\partial V$ from both sides.  
Thus $Z_V$ has no well-defined circular order. However, one can traverse $Z_V$ along the \textit{ideal boundary} of the disk $D_V$. We say that a \emph{parametrization of the ideal boundary of $D_V$}, or for short an \emph{ideal boundary parametrization of $D_V$}, is a piecewise analytic surjection $\gamma_V$ of the circle $\Circle$ onto $Z_V$ with the property that it is an immersion over preimages of edges, and that every edge is traversed at most once in each direction. (Equivalently, for every edge $e$ the preimage $\gamma_V^{-1}(e)$ consists of one or two intervals in $\Circle$, and so that the restriction of $\gamma_V$ to each of these intervals is a diffeomorphism onto the entire edge. Observe that an edge is traversed twice exactly if it bounds $V$ on both sides.) 
We will also assume that the orientation of an ideal boundary parametrization is inherited from the orientation of $V\subset\C$, so the winding number of the ideal boundary parametrization with respect to any point in $D_V$ is $1$. 

 Let $W$ be a component of $f^{-1}(V)$, and $\gamma_W \colon \Circle \to Z_W$ be an ideal boundary parametrization of $D_W$. The proper map $f \colon W \to V$ gives rise to a proper continuous orientation-preserving map $\tau \colon \Circle \to \Circle$ such that $f \circ \gamma_W  = \gamma_V \circ \tau$. We will refer to this map as the \textit{map between ideal boundary parametrizations}. Then $\tau$ is an orientation preserving degree $\delta$ self-cover of $\Circle$, where $\delta \geqslant 1$ denotes the degree of the proper map $f\colon W\to V$. There are $\delta$ choices for $\tau$.

To remedy the issue of pinching at fixed and pre-fixed points along the outer boundaries we will use the following ``thickening'' construction of $\Delta_n$.

First consider a closed set $\Delta^*\supset\Delta$ that contains a neighborhood of $\Delta\cap\C$ so that $f(\Delta^*)\subset\Delta^*$ and $f(\partial\Delta^*)\cap\partial\Delta^*=\{\infty\}$. To construct $\Delta^*$, recall that every edge of $\Delta$ connects $\infty$ to a fixed point $\xi\in\C$ of $f$ and corresponds to a radius in $\disk$ under an inverse Riemann map $\phi\colon\disk\to U_\xi$ with $\phi(0) = \xi$. We then thicken the radius in $\disk$ in a forward invariant way as indicated in Figure~\ref{Fig:ThickenedGraph}. Let $\Delta^*$ be the union of thickened radii for all edges in $\Delta$.

\begin{figure}[htbp]
\includegraphics[trim=20 20 20 20,scale=.6]{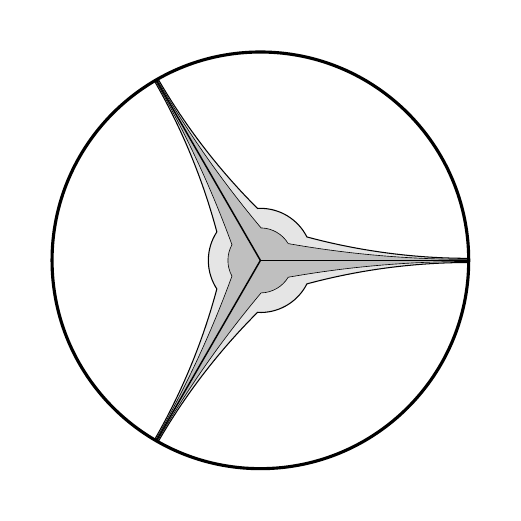}
\caption{The disk $\disk$ represents the immediate basin $U_\xi$ with respect to an inverse Riemann map $\phi\colon \disk\to U_\xi$ with $\phi(0)=\xi$. In the shown case, $U_\xi$ has three accesses to $\infty$; these correspond to three fixed rays in $\disk$. The thickened neighborhood of these rays, transported to $\disk$ by $\phi^{-1}$, consists of a small disk around $0$ and forward invariant neighborhoods of the rays (the union of all shaded domains). The image of the thickened neighborhood is shown in darker gray. 
}
\label{Fig:ThickenedGraph}
\end{figure}

In analogy to $\Delta_n$, which is the component of $f^{-n}(\Delta)$ containing $\infty$, we define $\Delta^*_n$ as the component of $f^{-n}(\Delta^*)$ containing $\infty$. In the initial construction of $\Delta^*$, there is a choice of the ``thickness'' of the thickening, and it has to be small enough so that $\Delta^*_n$ does not connect $\Delta_n$ to different components of $f^{-n}(\Delta)$; moreover, $\Delta^*_n$ should not contain any critical values that are not already in $\Delta_n$. 
Since $\Delta^*$ is forward invariant, we have $f(\Delta_n^*)=\Delta_{n-1}^*\subset \Delta_n^*$ , and the construction implies that $\partial\Delta^*_n\cap \Delta^*_{n-1}$ might consist only of $\infty$, and possibly poles and prepoles (all $\Delta^*_n$ are closed in $\Cc$).  Note that  for every component of $\Cc\sm\Delta^*_n$ the boundary is a simple closed curve, except that it may pass finitely many times through $\infty$ or through poles or prepoles (but only finitely many poles or prepoles because $\Delta_n$ only has finitely many vertices).

To accomplish the proof of Proposition~\ref{Prop_PoleNotConnected} we will need to show that $V_1$ and certain of its iterated preimages are unbounded. We will get this from a more general lemma that gives us control on the mapping degree of the map between ideal boundary parametrizations. The proof adopts ideas from the proof of \cite[Theorem~4.8]{RS} (quoted above as Lemma~\ref{Lem_Lefschetz}).

\begin{lemma}[Ideal boundary mapping degree]
\label{Lem_OuterDeg}\label{Lem_V1}
Let $f$ be a Newton map and $V$ be an unbounded component of $\Cc \sm f^{-n}(\Delta)$. Suppose $W$ is a component of $f^{-1}(V)$ with $W \subset V$, and let $\tau \colon \Circle \to \Circle$ be the map between the ideal boundary parametrizations of $D_W$ and $D_V$. Then $W$ is unbounded and has $\deg \tau $ distinct accesses to $\infty$. 
\end{lemma}

\begin{proof}
We use the thickened graphs $\Delta_n^*$ and set $V^* := V \sm \Delta_n^*$. By construction, $Z_{V^*} \cap \C$ contains no fixed points or critical values, and $\infty$ is an accessible boundary point of $V^*$ through exactly same accesses as in $V$. The boundary $\partial V^*$ is not a simple curve: it has a finite number of intersections at $\infty$, poles or prepoles.

Let $D\subset \Cc$ be a (small) open neighborhood of $\infty$ such that $D\cap V^*$ has exactly one component in each access of $V^*$ to $\infty$, and so that $\partial D$ is a piecewise analytic curve. Set $V' := V^* \sm D$; this is a bounded domain in $\C$.

Let $W'$ be the unique component of $f^{-1}\left(V'\right)$ intersecting $W$. This may or may not be a topological disk. Note that, by construction, the outer boundary $Z_{W'}$ is a simple closed curve except possibly at prepoles (but not poles). Let $\gamma \colon \Circle \to Z_{W'}$ be an ideal boundary parametrization of $D_{W'}$, $\gamma_0 \colon \Circle \to Z_{V'}$ be an ideal boundary parametrization of $D_{V'}$, 
and $\tau' \colon \Circle \to \Circle$ be the corresponding map  
between the ideal boundary parametrizations. By choosing the disk $D$, as well as the thickening, small enough, we have
\begin{equation}
\label{Eq:deg1}
\deg \tau = \deg \tau'=: \delta.
\end{equation}

\begin{figure}[htbp]
\includegraphics[scale=.8]{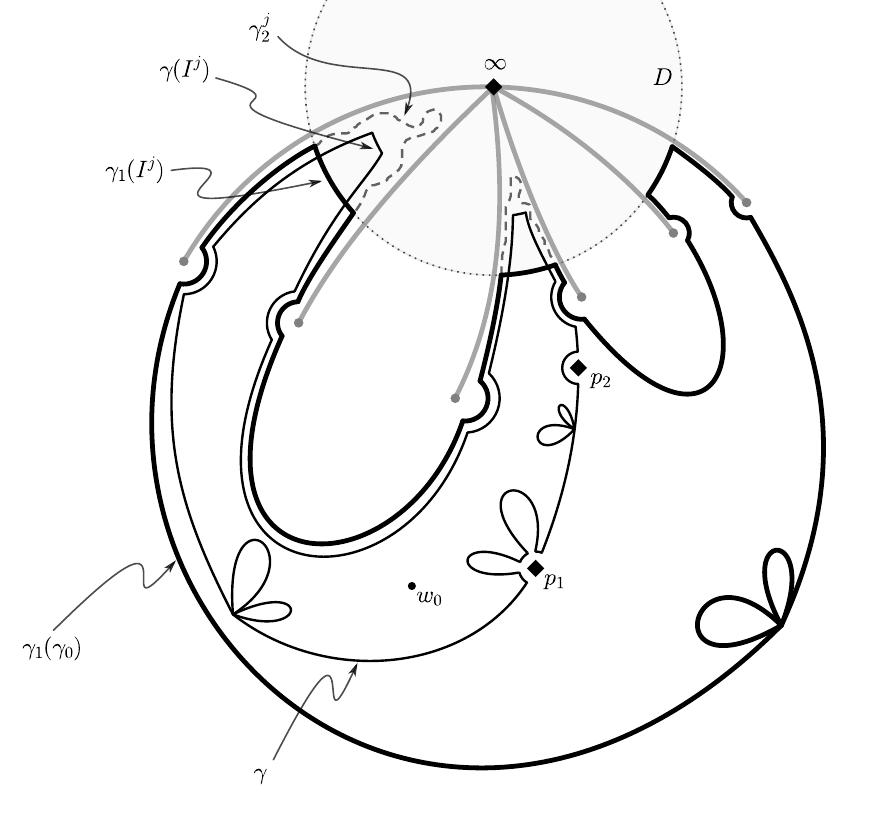}

\caption{The construction of curves in the proof of Lemma~\ref{Lem_OuterDeg}. The immersed circle $\gamma$ (which parametrizes the ideal boundary of $D_{W'}$) is shown in solid thin black; solid thick black lines represent the immersed circle $\gamma_0$ (the ideal boundary of $D_{V'}$) and $\gamma_1 = f \circ \gamma$; the curve $\gamma_1$ coils $\delta$ times along $\gamma_0$ (twice in this schematic example) and coincides with $\gamma_0$ as a set. The local modifications of $\gamma_1$ in each of the access to $\infty$ in $W$ (curves $\gamma_2^j$) are shown in dashed black. The thick gray lines represent the invariant rays $\Delta \cap Z_V$ connecting fixed points (gray dots) and $\infty$; $p_1$ and $p_2$ are poles on $Z_{W}$. On this picture $W$ has 2 accesses to $\infty$. 
}
\label{Fig_MapDeg}
\end{figure}

\textsc{Step 1.}
For $\gamma_1 := f \circ \gamma$ and every $w_0\in W'$, we have
\begin{equation}
\frac{1}{2\pi i}\oint \limits_{\gamma_1 - w_0} \frac{d \zeta}{\zeta} = \delta
\label{Eq:New1}
\end{equation}
because the winding number of the curve $\gamma_1$ around $w_0$ (the integral on the left) is exactly the mapping degree of $\tau' \colon \Circle \to \Circle$ (here we used~(\ref{Eq:deg1}) and the fact that an ideal boundary parametrization is a piecewise analytic immersion of the circle with winding number $1$).

\textsc{Step 2.} 
We want to count accesses to $\infty$ by (a variant of) an integral that counts fixed points of $f$, using the fact that $\infty$ is a fixed point of $f$, and the integral will do the count so that every access counts as a separate fixed point. We will employ a variant of Lemma~\ref{Lem_Lefschetz} that counts fixed points; that lemma requires a curve $\gamma$ so that $f\circ\gamma$ does not enter the domain bounded by $\gamma$. In our case, $\gamma$ bounds $ W'$, but $f\circ\gamma$ enters $W'$ once in every access, within $D$. (There are also some common boundary points of $W'$ and $V'=f(W')$ outside of $D$, but there are at prepoles and hence away from fixed points of $f$). 
We will remedy this issue by replacing $\gamma_1=f\circ\gamma$ by a new immersed circle $\gamma_2$ (which is not the image of $\gamma$ under $f$). 

If $W$ is bounded, then there is no issue and we simply set $\gamma_2 := \gamma_1$ and go to the passage after~(\ref{Eq:New6}) (which leads the boundedness assumption to a contradiction).

Let $m\ge 1$ be the number of accesses of $W$ to $\infty$, and enumerate the accesses from $0$ to $m-1$. There are $\delta m$ intervals $I^j\subset\Circle$ for which $\gamma_1\left(I^j\right)\subset\partial D$ (exactly $\delta$ intervals for which the image is within a given access to $\infty$); label them so that for $j \in\{0,1,\dots, \delta m - 1\}$ we have $\gamma_1(I^j)\subset\partial D$ within the $\lfloor j/\delta\rfloor$-th access of $W$. Pick some piecewise analytic Jordan arc $\gamma_2^j \colon I^j \to \ovl D$ joining the two ends of $\gamma_1(I^j)$ and avoiding $\Delta$  and the image of $\gamma$. Define an immersion $\gamma_2 \colon \Circle \to \C$ by putting
\begin{equation*}
\gamma_2(t) := \left\{
\begin{aligned}
&\gamma_2^j(t) \quad \text{ if }t \in I^j \text{ for some }j;\\
&\gamma_1(t) \quad \text{ otherwise}
\end{aligned}
\right.
\end{equation*}
(see Figure~\ref{Fig_MapDeg}). Without loss of generality we can assume that $\gamma_2^{j_1}\left(I^{j_1}\right)$ and $\gamma_2^{j_2}\left(I^{j_2}\right)$ coincide as sets for those $j_1$ and $j_2$ with $\lfloor j_1 /\delta\rfloor = \lfloor j_2 /\delta\rfloor$. Since $\gamma_2$ is clearly homotopic to $\gamma_1$ in $\C \sm \{w_0\}$, homotopy invariance of winding numbers yields (using~(\ref{Eq:New1}))
\begin{equation}
\label{Eq:New2}
\frac{1}{2\pi i}\oint \limits_{\gamma_2 - w_0} \frac{d \zeta}{\zeta} = \frac{1}{2\pi i}\oint \limits_{\gamma_1 - w_0} \frac{d \zeta}{\zeta} = \delta.
\end{equation}

\textsc{Step 3}. 
The domain $D_{W'}$ is a topological disk, thus contractible. Therefore, the curve $\gamma$, being the ideal boundary  parametrization of $D_{W'}$, is homotopic within $\ovl{D_{W'}}$ to the point $w_0$. More precisely, there is a (differentiable) homotopy $G\colon \Circle\times[0,1]\to \ovl{D_{W'}}$ with $G(t,0) \equiv w_0$ and $G(t, 1)=\gamma(t) $ and so that $G(t,s)\in D_{W'}$ for all $t\in\Circle$ and all $s<1$. (Note that $W'$ is not necessarily compactly contained in $V^*$: their boundaries might intersect at some prepoles).  

Setting $\gamma^{\langle s\rangle}(t):=G(t,s)$, we obtain a family of curves  $\gamma^{\langle s\rangle}\colon \Circle\to \ovl{D_{W'}}$ for $s\in[0,1]$ and $\gamma^{\langle s\rangle}\colon \Circle\to {D_{W'}}$ for $s\in[0,1)$. We then have for $s\in[0,1)$

\begin{equation*}
\frac{1}{2\pi i} \oint\limits_{\gamma_2 - \gamma^{\langle s\rangle}} \frac{d \zeta}{\zeta} = \frac{1}{2\pi i} \oint\limits_{\gamma_2 - w_0} \frac{d \zeta}{\zeta} = \delta
\label{Eq:New3a}
\end{equation*}
because the left integral is well defined for all $s<1$ (the traces of $\gamma_2$ and $\gamma^{\langle s\rangle}$ are disjoint) and depends continuously on $s$ with values in $\N$. By continuity, the same then holds for $s=1$ (observe that $\gamma_2$ was constructed so that $\gamma_2(t)\neq\gamma(t)$ for all $t$). This gives us (using~\ref{Eq:New2}) 
\begin{equation}
\frac{1}{2\pi i} \oint\limits_{\gamma_2 - \gamma} \frac{d \zeta}{\zeta} = \frac{1}{2\pi i} \oint\limits_{\gamma_2 - w_0} \frac{d \zeta}{\zeta} = \delta.
\label{Eq:New3}
\end{equation}

\textsc{Step 4}.
Now we go back from $\gamma_2$ to $\gamma_1$, taking into account the changes. To do this, assume, up to local  affine reparametrization, that $I^j = [-1, 1]$, and $J^j := I^j_1 \sqcup I^j_2$ is a formal union of two identical copies of $I^j$. For each $j$, define a closed piecewise analytic curve $\gamma_3^j \colon J^j \to \ovl D \sm \{\infty\}$ by setting 
\begin{equation*}
\gamma_3^j(t) := \left\{
\begin{aligned}
\gamma_2^j(t), \text{ if }s \in I_1^j;\\
\gamma_1(-t), \text{ if }s \in I_2^j.
\end{aligned}
\right.
\end{equation*}

Note that $J^j$ is a topological circle, so $\gamma_3^j$ is a piecewise analytic circle map.
Using~(\ref{Eq:New3}) we then have
\begin{equation}
\label{Eq:New6}
\delta = \frac{1}{2\pi i} \oint\limits_{\gamma_2 - \gamma} \frac{d \zeta}{\zeta} = \frac{1}{2 \pi i} \oint\limits_{\gamma_1 - \gamma} \frac{d \zeta}{\zeta} + \sum \limits_{j=0}^{\delta m - 1} \frac{1}{2 \pi i} \oint\limits_{\left(\gamma_3^j - \gamma\right)|_{J^j}} \frac{d \zeta}{\zeta}
\;.
\end{equation}

The integral over $\gamma_1 - \gamma$ counts the number of fixed points in $\ovl{D_{W'}}$ (this can be easily seen by subdividing this component into small pieces; for details see, for example, \cite[Lemma~4.7]{RS}). 
 But we have no fixed points in this component, thus the integral over $\gamma_1 - \gamma$ in~(\ref{Eq:New6}) is zero. 

Now let us look at the integrals over $\gamma_3^j - \gamma$ in~(\ref{Eq:New6}). They split into two types: (A) $\gamma(I^j) \not\subset D$ (which means that $\gamma(I^j)$ lies in some small neighborhood of a pole $p \in Z_W$); and (B) $\gamma(I^j) \subset D$. 

The integrals of type (A) contribute $0$ to the sum. Indeed, $\gamma(J^j)$  is contractible along its image to a point that lies outside of $\ovl D$ (for small enough $D$). Therefore, integrals of type (A) compute the winding number of $\gamma_3^j$ in $\C$ with respect to some point outside the bounded component of $\C \sm \gamma_3^j(J^j)$. Such winding numbers are zero. On the other hand, each of the integrals of type (B) contributes exactly 1 to the sum for a similar reason. We have $m$ such integrals in total, one in each access.  

Summing up all together, we see that~(\ref{Eq:New6}) implies $\delta=m$ as claimed.
In particular, we have $1\le \delta=m$, so $W$ has at least one access and is thus unbounded.
\end{proof}

\begin{proof}[Proof of Proposition~\ref{Prop_PoleNotConnected}]
We argue by induction on $n$; we already have $V_0$ and $V_1$ so that $V_0\supset V_1$, both are unbounded (for $V_1$ this is Lemma~\ref{Lem_V1}), and $V_1$ is multiply connected and has $p_*$ lying on a bounded component of its boundary. 

For the inductive step, suppose we have unbounded components $V_0\supset V_1\supset\dots\supset V_{n-1}\supset V_n$ so that every $V_l$ with $l\ge 1$ is a multiply connected component of $f^{-1}(V_{l-1})$ with the pole $p_*$ lying on a bounded component of $\partial V_l$.

Since $V_n$ is unbounded, hence $\infty\in\partial V_n$, there is a component $V_{n+1}$ of $f^{-1}(V_n)$ with $p_*\in\partial V_{n+1}$; there might be several such components and we choose one as follows. 

Since $\Delta$ is forward invariant under $f$ and all edges to $\infty$ in each $\Delta_n$ are already in $\Delta$, the set $V_{n}$ contains an  access to $\infty$ that is also an access to $\infty$ in $V_{n-1}$. This implies that $V_n$ has an access to $p_*$ and we choose $V_{n+1}$ among the components of $f^{-1}(V_n)$ that has the same access to $p_*$. Therefore $V_{n+1}\cap V_n\neq\emptyset$, hence $V_{n+1}\subset V_n$. By Lemma~\ref{Lem_V1}, the set $V_{n+1}$ is unbounded. 

The points $\infty$ and $p_*$ are both in $\partial V_{n+1}\subset f^{-(n+1)}(\Delta)$. If they are in the same boundary component, then $p_*\in\Delta_{n+1}$ and we are done. Otherwise, $p_*$ is contained in a bounded component of $\partial V_{n+1}$, so $V_{n+1}$ separates $p_*$ from $\infty$. Note that in such case the choice for $V_{n+1}$ is unique and we keep the induction going.
\end{proof}

Finally, we are ready to prove the key ingredient for the proof of Theorem~\ref{Thm_PreimagesConnected} for attracting-critically-finite maps.

\begin{theorem}[Poles connect to $\infty$]
\label{Thm_PolesConnect}
The graph $\Delta_n$ contains all poles of $f$ for all sufficiently large $n$.
\end{theorem}

\begin{proof} 

We argue by contradiction. If the claim is false, then by Proposition~\ref{Prop_PoleNotConnected} there is a pole $p_*$ and a sequence of multiply connected unbounded components $V_n$ of $\Cc\sm f^{-n}(\Delta)$ so that $p_*$ is in a bounded component of $\partial V_n$ (in particular, $V_n$ separates $p_*$ from $\infty$). Denote by $Z_n$ the outer boundary of $V_n$. 


Since the sequence $(V_n)_{n=0}^\infty$ is nested, any $z_0 \in \C$ is either in $\ovl{V_n}$ for all $n$, or it leaves some $\ovl {V_{n_0}}$ and then $z_0 \not\in \ovl{V_n}$ for $n \geqslant n_0$. Therefore, since $f$ has only a finite number of fixed points and poles, there exists an index $n_0$ such that $\ovl{V_n}$ contains the same poles and the same fixed points  for all $n \geqslant n_0$. By construction, fixed points in $\ovl{V_n}$ can lie only in $Z_n$. In fact, each $Z_n$ contains $\infty$ and thus some number of edges in $\Delta$ starting at $\infty$, which must terminate at fixed points in $\C\cap Z_n$. At the same time, $Z_n$ might potentially contain fixed points that are not connected to $\infty$ by an edge in $Z_n$, and thus these fixed points are connected to a prepole by an edge in $Z_n$.
  
All $V_n$ have at least one access to $\infty$. The number of accesses to $\infty$ within $V_n$ cannot increase in $n$ because $V_{n+1}\subset V_n \subset \Cc \sm \Delta_n$ and all $\Delta_n$ have the same edges to $\infty$ (given by the channel diagram $\Delta$). 

Combining this, we obtain control on the structure of $V_n$ for large $n$. There is a minimal index, for which we keep the notation $n_0$, so that for all $n \geqslant n_0$ all $V_n$ have the same accesses to $\infty$, and all $Z_n$ contain the same fixed points and the same poles. 


For $n\ge n_0$, let $\gamma_n \colon \Circle \to \Cc$ be an ideal boundary parametrization of $D_{V_n}$; by definition $\gamma_n$ is a piecewise analytic surjection $\gamma_n\colon\Circle\to Z_n$ with winding number~$1$.  Let $\tau_n \colon \Circle \to \Circle$ be the corresponding map between the ideal boundary para\-me\-tri\-za\-tions of $D_{V_{n+1}}$ and $D_{V_{n}}$: this map satisfies $f \circ \gamma_{n+1} = \gamma_{n} \circ \tau_n$ for all $n$. Observe for later use that we are free to pre-compose $\tau_n$ by any (orientation preserving) piecewise analytic circle homeomorphism, at the expense of re-parametrizing $\gamma_{n+1}$ (which we are free to do). 

Our choice of the index $n_0$ above has natural implications on the behavior of the functions $\tau_n$ for all $n \geqslant n_0$. First of all, if $m \geqslant 1$ is the number of accesses to $\infty$ in $V_n$, which we know is constant for all $n \geqslant n_0$, then Lemma~\ref{Lem_OuterDeg} implies that $\tau_n$ is a covering map of degree exactly $m$ for all $n \geqslant n_0$. 
Moreover, the choice of $n_0$ implies that there exist $3m$ distinct points  $t_1, s_1, t'_1, t_2, s_2, t'_2,\dots, t_m, s_m, t'_m \in \Circle$, listed in cyclic order, so that $\{s_1, \ldots, s_m\} = \gamma_n^{-1}\left(\{\infty\}\right)$, the $\gamma_n(t_i)$ and $\gamma_n(t'_i)$ are the finite fixed points of $f$, and so that for all $i \in \{1, \ldots, m\}$ the restrictions of $\gamma_n$ to the intervals $[t_i, t'_i]$ are homeomorphisms to their images. The last assertion follows from the fact that $\gamma_n(t_i)\neq\gamma_n(t'_i)$ (otherwise, there would be either an edge from $\gamma_n(t_i)$ to $\infty$ that was followed directly by the same edge in backwards direction, or there would be two edges connecting $\gamma_n(t_i)$ to $\infty$ followed one after the other; both possibilities are clearly impossible: loops in the ideal boundary can only ``stick in to the inside'', not ``stick out to $\infty$'').

Moreover, we have $t'_i\neq t_{i+1}$ (but possibly $\gamma_n(t'_i)=\gamma_n(t_{i+1})$) because for $n\ge 1$ at every fixed point in $\C$ the graph $\Delta_n$ does not have adjacent edges to $\infty$. We then have intervals  $I_i := [t'_i, t_{i+1}] \subset \Circle$ for $i \in \{1, \ldots, m\}$ of positive length.

A priori, the $3m$ points $t_1, s_1, t'_1,\dots, t_m, s_m, t'_m$  depend on $n$, but since the fixed points along $Z_n$ are the same for $n \geqslant n_0$, a suitable piecewise linear reparametrization of $\tau_n$ as mentioned above can bring these points to identical positions for all $n$. Hence $\tau_n(s_i) = s_i$, $\tau_n(t_i) = t_i$, $\tau_n(t_i') = t_i'$ and $\tau_n \left([t_i, t_i']\right) = [t_i, t_i']$ for all $n \geqslant n_0$ and all $i$.
 Then every component of $\Circle\sm \bigcup I_i$ is mapped homeomorphically to itself under each $\tau_n$.

\begin{claim}
\label{Claim:Homeo}
There exists an interval $I_i$ such that  for all $n \geqslant n_0$ the image $\gamma_n(I_i)$ does not contain poles and $\tau_n$ restricted to ${I_i}$ is  an orientation-preserving homeomorphism to itself. 
\end{claim}

Let $J_n\subset\{1,\dots,m\}$ consist of those indices $j$ for which $\tau_n$ restricted to $I_j$ is not a homeomorphism. Since the closure of $\Circle \sm \bigcup I_i$ is fixed under $\tau_n$ and $\tau_n(s_i) = s_i$, it follows that for every $n>n_0$ and $j\in J_n$, the interval $I_j$ contains at least one point $p_{ij}$ with $\tau_n(p_{ij}) = s_i$. Therefore, $s_j$ has at least $|J_n|+1$ preimages under $\tau_n$ (the point $s_j$ as well as the points $p_{ij}\in I_j$), and thus the total mapping degree $\deg \tau_n$ is at least $|J_n|+1$. Since by Lemma~\ref{Lem_OuterDeg} this degree equals $m$, we have $|J_n|<m$, so for every $n$ there is at least one interval $I_i$ on which $\tau_n$ acts as homeomorphism. In fact, such indices $i$ do not depend on $n$ because the structure of $Z_n$ is the same for all $n \geqslant n_0$, so for fixed $i$ the curve $\gamma_n(I_i)$ cannot acquire additional poles as $n$ grows. This proves \textsc{Claim \ref{Claim:Homeo}}.


Up to renaming, assume $I_1 = [t'_1, t_2]$ is an interval as given by \textsc{Claim \ref{Claim:Homeo}}, and put $\xi_1 := \gamma_n(t'_1)$, $\xi_2 := \gamma_n(t_2)$ (possibly with $\xi_1 = \xi_2$). Let $U_1$ and $U_2$ be the immediate basins of $\xi_1$ and $\xi_2$ and set $\Gamma_1:=\gamma_n\left((s_1,t'_1)\right)$ and $\Gamma_2:=\gamma_n\left((t_2,s'_2)\right)$ (these are internal rays in $U_1$ resp.\ $U_2$, fixed by $f$, that connect $\xi_1$ or $\xi_2$  to $\infty$, and hence do not depend on $n$).

\begin{claim}
\label{Claim:Bridge}
$\gamma_{n}\left(I_1 \sm\{t'_1, t_2\}\right)$ does not contain fixed points of $f$ for all $n \geqslant n_0$.
\end{claim}

Suppose to the contrary that $\gamma_n(I_1 \sm\{t'_1, t_2\})$ contains a fixed point of $f$, say $\xi$, for all $n \geqslant n_0$ (the cases $\xi = \xi_1$ or $\xi = \xi_2$ are possible). 
Every access to $\xi$ within $V_n$ corresponds to a unique value $v_{n,i} \in I_1$ with $\gamma_n(v_{n,i})=\xi$, and there are finitely many such values because $Z_n$ is a finite graph.

Since $V_{n+1}\subset V_n$, every access to $\xi$ in $V_{n+1}$ is contained in an access within $V_n$. Moreover, every access to $\xi$ in $V_{n+1}$ must map under $f$ to an access to $\xi$ in $V_n$, and this must happen so as to preserve the order of the corresponding points $v_{n+1,i}$ and $v_{n,i'}$ on $I_1$ under $\tau_n$.

This implies that there is an access to $\xi$ within $V_{n+1}$ that is contained in the access within $V_n$ that it maps to under $f$. These accesses are bounded by edges in $Z_n$. Now we use again an inverse Riemann map $\phi \colon \disk\to U_\xi$ with $\phi(0)=\xi$ and its induced self-map $h:=\phi^{-1}\circ f\circ \phi\colon \disk\to\disk$; it has the form $h(z)=\lambda z^k$ with $|\lambda|=1$ and $k\ge 2$. Any edge in $Z_n$ terminating at $\xi$ corresponds under $\phi^{-1}$ to a radius in $\disk$. Therefore, the two accesses to $\xi$ within $V_{n+1}$ resp.\ $V_n$ correspond to sectors $S_{n+1}$ and $S_n$, both bounded by two radii each, and satisfying $S_{n+1}\subset S_n$ and $h(S_{n+1})=S_n$. Since the action of $h$ on $\partial\disk$ is an expanding circle endomorphism of degree $k\ge 2$, this implies that $S_{n+1}$ contains a  radius that is fixed by $h$. Therefore, $V_n$ contains an edge that is fixed by $f$; this must be an edge from $\xi$ to $\infty$, a contradiction. This proves \textsc{Claim \ref{Claim:Bridge}}.

\textsc{Claim~\ref{Claim:Bridge}} implies that for all $n \geqslant n_0$ the set $\gamma_n \left(I_1 \sm \{t_1', t_2\}\right)$ contains a non-trivial Jordan arc (without endpoints) connecting $\xi_1$ to $\xi_2$ that avoids fixed points of $f$; call this arc a \emph{bridge between $\xi_1$ and $\xi_2$} and denote it by $X_n$. If $\xi_1=\xi_2$, then this bridge is a Jordan curve with the point $\xi_1 = \xi_2$ missing. If such a bridge exists, it is unique for a given $n$. By \textsc{Claim~\ref{Claim:Homeo}}, $f$ maps the bridge $X_{n+1}$ homeomorphically onto $X_n$, and it maps its edges to edges and vertices to vertices in a one-to-one fashion; in particular all $X_n$ have the same number of edges, and the same number of vertices.
Moreover, no bridge contains a fixed point. This implies that every vertex can be in only finitely many bridges. 

The two arcs $\ovl{X}_n$ and $\ovl\Gamma_1 \cup \ovl\Gamma_2$ are both simple arcs that connect $\xi_1$ to $\xi_2$, and they are are disjoint except for their endpoints. Therefore, $\ovl X_n \cup \ovl\Gamma_1 \cup \ovl\Gamma_2=\Gamma_1\cup\{\xi_1\}\cup X_n\cup\{\xi_2\}\cup\Gamma_2\cup\{\infty\}$ is a Jordan curve in $Z_n$. 

Another consequence of \textsc{Claim \ref{Claim:Bridge}} is that all $V_n\cap U_1$ and $V_n\cap U_2$ are connected: any connected component must be bounded by two edges in $Z_n$, and for $U_1$ two such edges are $\Gamma_1$ and the first edge in $X_n$, and no further edges in $\ovl X_n \cup \ovl\Gamma_1 \cup \ovl\Gamma_2$; any additional edges would disconnect $V_n$; similar for $U_2$. 

The bridge $X_n\subset\Delta_n$ is finite graph passing through a chain of Fatou components from the attracting basins, where each of the components in the chain shares a pole or a prepole on its boundary with the adjacent component(s) on the chain. Thus vertices along a bridge that are not endpoints are alternating (pre-)poles and preimages of the finite fixed points of $f$.

By our standing assumption, the pole $p_*$ never belongs to $Z_n$. From here we want to arrive at the final contradiction by employing a ``sweeping argument'' as follows.

\begin{claim}
\label{Claim:Nest}
If $W_n$ is the component of  $\Cc \sm \left(\ovl X_n \cup \ovl\Gamma_1 \cup \ovl\Gamma_2\right)$ containing $V_n$, then $\left(\ovl W_n\right)_{n\geqslant n_0}$ is a nested sequence of closed topological disks and 
\begin{equation}
\label{sweeping}
\bigcap \limits_{n=n_0}^\infty \ovl W_n = \ovl\Gamma_1 \cup \ovl\Gamma_2.
\end{equation}
\end{claim}

Since (\ref{sweeping}) is impossible when $\partial V_n$ has a bounded component disjoint from $Z_n$ and containing $p_*$, \textsc{Claim~\ref{Claim:Nest}} will give us the final contradiction that yields the conclusion of Theorem~\ref{Thm_PolesConnect}.


%

To prove the claim, assume that the inverse Riemann map $\phi_i\colon\disk\to U_i$ with $\phi_i(0)=\xi_i$ and $\phi_i^{-1}\circ f \circ \phi_i(z)=z^{k_i}$ is rotated so that $\phi_i^{-1}(\Gamma_i)$ is the radius $(0,1)$ of the disk $\disk$.

Let $e^1_n := X_n \cap U_1$ be the edge in $X_n$ that connects $\xi_1$ to a (pre-)pole $r^1_n\in\partial U_1$. The edge $e^1_n$ corresponds to some ray $(0,1)\cdot \exp(2\pi i \alpha)$ in $\disk$. Since $f(X_{n+1})=X_n$, we have $f(e^1_{n+1})=e_n^1$, and since $V_{n+1}\subset V_n$, we have $e^1_{n+1} = \phi_1 \left((0,1)\cdot \exp(2\pi i \alpha/k_1)\right)$.
Analogous constructions and results hold for the edge $e^2_n := X_n \cap U_2$ that connects $\xi_2$ to $r^2_n \in \partial U_2$.

For a given $i \in \{1,2\}$, the endpoints $r_n^i$ of $e_n^i$ satisfy $r_n^i\to\infty$ as $n \to \infty$. Indeed, connect $r_n^i$ to $\infty$ by injective curves $\sigma_n^i\subset U_i\sm\{\xi_i\}$; for simplicity, use hyperbolic geodesics within $(W_n\cap U_i) \sm \{\xi_i\}$ so that $f(\sigma_{n+1}^i)=\sigma_n^i$ (these can be constructed inductively by appropriate preimages of $f$). Since $\infty$ is a repelling fixed point, the spherical lengths of all $\sigma_n^i$ are finite (the curves converge to $\infty$ geometrically), and as $n\to\infty$ these lengths tend to $0$. In particular $r_n^i\to\infty$ as $n\to\infty$.

Define $L_n:=X_n \sm (e_n^1 \cup e_n^2)$. By construction, $L_n$ is either a point or a simple Jordan arc (closed as a set). In the first case this implies (\ref{sweeping}) as follows: the curves $\sigma_n^i$ cut $U_i$ into two pieces, and hence the simple closed curves $\ovl{\sigma_n^1\cup\sigma_n^2}$ subdivide $\ovl W_n$ into three components: one is surrounded by $\ovl{\sigma_n^1\cup\sigma_n^2}$, and the other two are within $U_1$ and $U_2$, respectively. Therefore, any point in $\bigcap_{n=n_0}^\infty \ovl W_n$ is in $U_1$ or $U_2$, or is contained in, or surrounded by all the Jordan curves $\ovl{\sigma_n^1\cup\sigma_n^2}=\sigma_n^1\cup\sigma_n^2\cup \{r_n^1,\infty\}$. But the spherical diameters of $\ovl{\sigma_n^1\cup\sigma_n^2}$ tend to $0$, so these curves shrink to the point $\{\infty\}$, and $\ovl {W_n}\cap U_1$ shrinks to $\Gamma_1 \cup \{\xi_1\}$ as $n\to\infty$, and similarly for $U_2$. This proves (\ref{sweeping}) when $r_n^1=r_n^2$.

If not, we will show that the spherical lengths of $L_n$ tend to zero, and this will prove the claim in the same way, using the Jordan curves $\sigma_n^1\cup\sigma_n^2\cup\{\infty\}\cup \ovl L_n$. 

To begin with, $f$ maps $L_{n+1}$ homeomorphically onto $L_n$, and $(L_n)_{n\ge n_0}$ is a sequence of (not necessarily disjoint) arcs  composed of the same number of edges of $\Delta_n$. The Jordan arc $L_n$ contains finitely many (pre)poles. Label them by  $r_n(1), \ldots, r_n(l)$ in the natural order, starting with $r_n(1) = r_n^1$ and terminating with $r_n(l) = r_n^2$. Note that $L_{n+1}$ also contains exactly $l$ prepoles, with $f\left(r_{n+1}(i)\right) = r_n(i)$ for all $i$.

Consider small neighborhoods around all $r_n(i)$ on $L_n$, and small neighborhoods of the components of $L_n \sm \bigcup_{i=1}^k r_n(k)$ in their respective Fatou components (we clarify the meaning of ``small'' below). Let $U(L_n)$ be the union of both types of neighborhoods. Moreover, we can choose these neighborhoods so that $f \colon U(L_{n+1}) \to U(L_n)$ is a biholomorphic map for all $n$ larger than some $n_1 \geqslant n_0$. Indeed, for $n$ sufficiently large $f \colon L_{n+1} \to L_n$ is a homeomorphism that avoids any postcritical points (by \textsc{Claim \ref{Claim:Bridge}} every bridge, and thus $L_n$, avoids fixed critical points, while a non-fixed critical point can stay on $L_n$ only for finitely many $n$ and then never returns).  Then construct $U(L_n)$ small enough so as to avoid critical values, and construct $U(L_{n+1})$ as the preimage component of $U(L_n)$ around $L_{n+1}$, and continue inductively; any time a critical value is encountered, we may have to shrink these neighborhoods, but this happens at most once for each of the finitely many critical values. (We do not claim, nor use, the fact that $U(L_n)$ for large $n$ avoids postcritical points.)

By construction, $U(L_n)$ is a topological disk. Each $U(L_n)$ carries its own normalized hyperbolic metric, and all $L_n$ have the same hyperbolic length with respect to it. The two endpoints of $L_n$ are the points $r_n^1$ and $r_n^2$ that converge to $\infty\not\in U(L_n)$. Therefore, the endpoints of $L_n$ are close to $\partial U(L_n)$ in the spherical metric. This implies that the spherical lengths of $L_n$ converge to zero, and as indicated above this proves \textsc{Claim \ref{Claim:Nest}} and also the theorem.
\end{proof}

\begin{corollary}[Prepoles connect to $\infty$]
\label{Cor:PrepolesConnect}
Every prepole of $f$ is contained in $\Delta_n$ for sufficiently large $n$. 
\end{corollary}
\begin{proof}
By Theorem~\ref{Thm_PolesConnect}, there is an index $n$ so that $\Delta_{n}$ contains all poles and hence connects all poles to $\infty$. Then $\Delta_{n+1}$ connects every preimage of a pole to a pole, and hence also to $\infty$, and by induction $\Delta_{n+k}$ connects every prepole in $f^{-(k+1)}(\infty)$ to $\infty$. 
\end{proof}

Note that in general $f^{-n}(\Delta)$ consists of an increasing number of components; every particular one is in $\Delta_{n'}$ for sufficiently large $n'$, but the number of components still tends to $\infty$.

We can now finally complete the proof of Theorem~\ref{Thm_PreimagesConnected}: it says that every Fatou component in the basin of any root can be connected to $\infty$ by a finite chain of other Fatou components in the basins of roots, so that adjacent components are connected by common boundary points that are poles or prepoles.

\begin{proof}[Proof of Theorem~\ref{Thm_PreimagesConnected}]
By definition, a Newton map $f$ is attracting-critically-finite if every critical point of $f$ in the basin $B_\xi$ of a root $\xi$ lands at $\xi$ after finitely many iterations. We first prove the theorem in the special case that $f$ is attracting-critically-finite, and then reduce the general case to this case using standard quasiconformal surgery.

Suppose that $f$ is attracting-critically-finite. 
Every component of every $B_{\xi_i}$ is an iterated preimage of the immediate basin $U_{\xi_i}$, so this component contains a pole or a prepole on its boundary, and this prepole can be connected to $\infty$ by some appropriate $\Delta_n$. This $\Delta_n$ is a finite union of edges, each of which is a curve in some component of $B_{\xi_i}$ that terminates at a prepole. This implies the claim.

One can turn any Newton map $f_0$ into an attracting-critically-finite Newton map $f_1$ by a quasiconformal surgery procedure that was pioneered by Shishikura \cite{Shishikura} and that is now standard (compare the exposition by Branner and Fagella \cite{BrannerFagella}): one needs to change the dynamics on compact subsets of finitely many components of the $B_\xi$ and glue the new dynamics in by a quasiconformal map. All fixed  points in $\C$ of $f_1$ are superattracting, so by Proposition~\ref{Prop_Head} the new map $f_1$ is a Newton map as well.

The old and new maps are topologically conjugate on their Julia sets and the conjugation respects Fatou components and basins; in particular, the connection between components of the basins required in the theorem is unchanged by the surgery or its inverse. We proved that this connection exists for $f_1$, so it must have existed before for $f_0$. Therefore, Theorem~\ref{Thm_PreimagesConnected} holds for every Newton map, attracting-critically-finite or not.
\end{proof}

\begin{remark}
The curve $\gamma$ that connects prepoles to $\infty$ can, in the attracting-critically-finite case, be chosen in a preferred way: in this case, all immediate basins $U_\xi$ have B\"ottcher coordinates, and prepoles on $\partial U_\xi$ can be connected to $\xi$ by internal rays. These internal rays also exist in all components of the basin of $\xi$ and provide the preferred connection for $\gamma$. For arbitrary Newton maps, the curves $\gamma$ still exist as required, but there is no such obvious choice of connections (one possibility for a preferred choice of $\gamma$ is to require that it intersects each component of $B_\xi$ in a curve that is a subset of a hyperbolic geodesic, but this is not invariant under the dynamics). 
\end{remark}

\section{The Newton Graph of a Newton Map}
\label{Sec_Graph}

In this section, we define abstract Newton graphs and use Theorem \ref{Thm_PreimagesConnected} to show that every postcritically fixed Newton map $f$ generates a unique abstract Newton graph in a natural way. From now on, we only work with postcritically fixed Newton maps; these are attracting-critically-finite by definition.

In this section, we construct finite graphs and, from these, postcritically finite branched covers that model the dynamics of postcritically fixed Newton maps. In Section~\ref{Sec_Thurston} we will then show, using Thurston theory, that these branched covers are actually realized as Newton maps of polynomials. The basic concepts of Thurston theory are reviewed in Section~\ref{Sec_Thurston}. In preparation for that section, we will use the concept of \emph{Thurston equivalence} several times here; a formal definition is given in Definition~\ref{Def:ThurstonEquivalence}.

\subsection{Extending Maps on Finite Graphs}
The channel diagram motivates the definition of a Newton graph. For this, we first need to introduce some notation regarding maps on embedded graphs and their extensions to $\S^2$, compare \cite[Chapter 6]{BFH}. We assume in the following that all graphs are embedded in $\S^2$.

\begin{definition}[Graph map, graph homeomorphism]
Let $\Gamma_1,\Gamma_2\subset\S^2$ be two finite graphs. We call a continuous map $g:\Gamma_1\to\Gamma_2$ a \emph{graph map} if it is injective on each edge of $\Gamma_1$, if forward and inverse images of vertices are vertices, and if the map is compatible with the embedding into $\Sphere$. If the graph map $g$ is a homeomorphism, then we call it a \emph{graph homeomorphism}. 
\end{definition}

The condition that $g$ is compatible with the embedding into $\Sphere$ is a local condition at every vertex $v$: if $g$ is locally injective at $v$, then $g$ should preserve the cyclic order of edges at $v$; if not, then the cyclic order of edges at $v$ and $g(v)$ should be compatible with a local covering of degree $\deg_vg$. 
We will use this definition in the case $\deg_vg>1$ only when the number of edges at $v$ equals $\deg_v g$ times the number of edges at $g(v)$ (compare Definition~\ref{Def_NewtonGraph} condition \eqref{Cond_Saturated}); in this case, the edges at $v$ are arranged in the same cyclic order as the image edges at $g(v)$, but repeated $\deg_vg$ times.

\begin{definition}[Regular extension]
\label{Def_GraphMap}
Let $g:\Gamma_1\to\Gamma_2$ be a graph map. An orientation-preserving branched covering map $\ol{g}:\S^2\to\S^2$ is called a \emph{regular extension} of $g$ if $\ol{g}|_{\Gamma_1}=g$ and $\ol{g}$ is injective on each component of $\S^2\setminus \Gamma_1$.
\end{definition}

Clearly, if a regular extension exists, it is unique up to homotopy of $\S^2$ relative to $\Gamma_1$; in particular, this will imply that it is unique up to Thurston equivalence.

For a vertex $v\in\Gamma_1$, we define the \emph{degree of $g$ at $v$}, denoted $\deg_v g$, as the maximal number of edges at $v$ that are mapped to the same image edge at $g(v)$. If $g$ has a regular extension $\ovl g$ and $\Gamma_1=\ovl g^{-1}(\Gamma_2)$, then $\deg_v g$ equals the local mapping degree of $\ovl g$ as a branched cover.

Now we discuss the existence of regular extensions.
Let $g:\Gamma_1\to\Gamma_2$ be a graph map. We may assume without loss of generality that each vertex $v$ of $\Gamma_1$ has a neighborhood $U_v\subset\S^2$ such that all edges of $\Gamma_1$ that enter $U_v$ terminate at $v$, and that these neighborhoods have disjoint closures for different vertices. We may also assume that for appropriate local coordinates $U_v$ is a round disk of radius $r$ centered at $v$ and that all edges entering $U_v$ are straight radii in $U_v$. We make analogous assumptions for $\Gamma_2$, with the same radius $r$. Finally, we may assume that $g|_{U_v}$ preserves lengths on the edges within $U_v$.  Then we can extend $g$ to each $U_v$ as in \cite{BFH}: for a vertex $v\in\Gamma_1$, let $\gamma_1$ and $\gamma_2$ be two adjacent edges ending there. In polar coordinates centered at $v$, these are straight lines with arguments, say, $\theta_1,\theta_2$ such that $0<\theta_2-\theta_1 \leqslant 2\pi$ (if $v$ is an endpoint of $\Gamma_1$, hence the endpoint of a single edge, then set $\theta_1=0$, $\theta_2=2\pi$). In the same way, choose arguments $\theta_1',\theta_2'$ for the image edges in $U_{g(v)}$ and extend $g$ to a (local) map $\tilde{g}$ on $\Gamma_1\cup\bigcup_v U_v$ by setting
\begin{equation}
\tilde g\colon(\rho,\theta)\mapsto \left(\rho, \frac{\theta_2'-\theta_1'}{\theta_2-\theta_1}\cdot(\theta-
\theta_1)+\theta'_1\right),
\label{Eq:LocalExtensions}
\end{equation}
where $(\rho,\theta)$ are polar coordinates in the sector bounded by the rays at $\theta_1$ and $\theta_2$. This way, sectors are mapped onto sectors in an orientation-preserving way. Then we have the following result which says that the only possible obstruction to the existence of a regular extension is that within some component of $\S^2\sm\Gamma_1$ several vertices have the same image. 

\begin{proposition}[Regular extension, {\cite[Proposition 6.4]{BFH}}]
\label{Prop_RegExt}
The map $g:\Gamma_1\to\Gamma_2$ has a regular extension if and only if for every vertex $y\in\Gamma_2$ and every component\/ $U$ of $\S^2\setminus\Gamma_1$, the local extension $\tilde{g}$ from \eqref{Eq:LocalExtensions} is injective on
\[
    \bigcup_{v\in g^{-1}(y)} U_v \cap U\;.
\]
If the regular extension exists, it has critical points only at the vertices of $\Gamma_1$.
\qed
\end{proposition}


\subsection{The Newton Graph}
Let us first define an abstract channel diagram; it plays a fundamental role for the main concept of an abstract Newton graph.

\begin{definition}[Abstract channel diagram]
\label{Def_ChannelDiagram} An \emph{abstract channel diagram of degree $d \geqslant 3$} is a graph $\Deltahat \subset \S^2$ with exactly $d+1$ vertices $v_0,\dots,v_d$ and with edges $e_1,\dots,e_l$ so that the following properties are satisfied:
\begin{enumerate}
\item 
\label{I:AbsCh1}
$l \leqslant 2d-2$;
\item 
\label{I:AbsCh2}
each edge joins $v_0$ to some $v_i$ with $i>0$;
\item 
\label{I:AbsCh3}
each $v_i$ is connected to $v_0$ by at least one edge;
\item 
\label{necessaryCondition} 
if $e_i$ and $e_j$ both join $v_0$ to $v_k$, then each connected component of
$\S^2\setminus \ol{e_i\cup e_j}$ contains at least one vertex of $\Deltahat$.
\end{enumerate}
\end{definition}

We say that an abstract channel diagram $\Deltahat$ is \emph{realized by a Newton map with channel diagram ${\Delta}$} if there exists a graph homeomorphism $h:\Deltahat\to{\Delta}$ that preserves the cyclic order of edges at each vertex (this cyclic order is well defined because $\Deltahat$ is embedded in $\S^2$).

\begin{lemma}[Abstract channel diagram]
\label{Lem:AbstractChannelDiagram}
The channel diagram ${\Delta}$ of the Newton map $f$ constructed in Section \ref{Sec_Models} is an abstract channel diagram.
\end{lemma}
\begin{proof}
The channel diagram $\Delta$ of a Newton map is constructed as the union of all fixed internal rays within immediate basins (see Section~\ref{Sec_Models} and in particular equation \eqref{Eq:DefChannelDiagram}); these represent accesses to $\infty$ (channels). Therefore, with $v_0$ standing for $\infty$ and the rest of $v_i$ standing for the finite fixed points, properties (\ref{I:AbsCh2}) and (\ref{I:AbsCh3}) of Definition \ref{Def_ChannelDiagram} are immediate. Since for every immediate basin the number of channels equals the number of critical points it contains, the total number of channels equals the number of critical points in all the immediate basins combined, which implies (\ref{I:AbsCh1}). 

Finally, ${\Delta}$ satisfies (\ref{necessaryCondition}) because for every immediate basin $U_{\xi}$ of $f$, every component of $\C\setminus U_\xi$ contains at least one fixed point of $f$ \cite[Corollary~5.2]{RS} (see also Theorem \ref{Thm_FixedPoles}).
\end{proof}

With these preparations, we are ready to introduce the concept of an abstract Newton graph. It turns out that it carries enough information to uniquely characterize postcritically fixed Newton maps.

\begin{definition}[Abstract Newton graph]
\label{Def_NewtonGraph}
Let $\Gamma\subset\S^2$ be a finite connected graph, $\Gamma'$ the set of its vertices and $g:\Gamma\to\Gamma$ a graph map. The pair $(\Gamma,g)$ is called
an \emph{abstract Newton graph} if it satisfies the following conditions:
\begin{enumerate}
\item 
\label{Cond_Diagram} There exists $d_{\Gamma}\geqslant 3$ and an abstract channel diagram $\Deltahat\subsetneq\Gamma$ of degree $d_\Gamma$ such that
$g$ fixes each vertex and each edge of $\Deltahat$.
\item  
\label{Cond_Branch} If $v_0,\dots,v_{d_\Gamma}$ are the vertices of $\Deltahat$, then $v_i\in \ol{\Gamma\setminus\Deltahat}$ if and only if $i\neq 0$.
Moreover, there are exactly $\deg_{v_i} g - 1 \geqslant 1$ edges in $\Deltahat$ that connect $v_i$ to $v_0$ for $i\neq 0$. We call $v_0$ the \emph{distinguished vertex of $\Gamma$}.
\item 
\label{Cond_Degree} $\sum_{v \in\Gamma'} \left(\deg_v g - 1\right) = 2d_\Gamma - 2$.
\item 
\label{Cond_Enter} There is an $n\in\N$ such that $g^{\circ (n-1)}(v)\in\Deltahat$ for all $v\in\Gamma'$ with $\deg_v g>1$. Let $N_\Gamma$ be the least such $n$. 
\item

\label{Cond_Preimages}		
Every vertex $v$ and every edge $e$ of $\Gamma$ satisfy $g^{\circ N_\Gamma}(v)\in\hat\Delta$ and $g^{\circ N_\Gamma}(e)\subset\hat\Delta$.	

\item
\label{Cond_Saturated}
For every $v\in\Gamma'\sm\Deltahat$ with  $g^{\circ (N_\Gamma-1)}(v)\in\Deltahat$, the number of adjacent edges in $\Gamma$ equals $\deg_v g$ times the number of edges adjacent to $g(v)$. 
\item 
\label{Cond_Connected} The graph $\ol{\Gamma\setminus\Deltahat}$ is connected.
\item 
\label{Cond_Extension} For every vertex $y\in\Gamma'$ and every component $U$ of $\S^2\setminus\Gamma$, the local extension $\tilde{g}$ from \eqref{Eq:LocalExtensions} is injective on 
\( 
\bigcup_{v\in g^{-1}(y)} U_v \cap U\;.
\)
\end{enumerate}
\end{definition}

Conditions~\eqref{Cond_Diagram} and \eqref{Cond_Branch} make sure that the abstract Newton graph contains an abstract channel diagram with enough edges connecting the finite fixed points to $\infty$, and condition~\eqref{Cond_Degree} guarantees that there are enough critical points. Conditions~\eqref{Cond_Enter} and \eqref{Cond_Preimages} require that all vertices and edges are preimages of the abstract channel diagram and prescribe the allowed number of iterations. Condition~\eqref{Cond_Saturated} is a saturation condition: every vertex up to generation $N_\Gamma-1$ has enough edges attached (this cannot be true for generation $N_\Gamma$ as well, or we would obtain vertices of generation $N_\Gamma+1$ and thus an infinite graph). Condition~\eqref{Cond_Connected} is a non-obvious but important property of Newton graphs of actual Newton maps. Finally, condition~\eqref{Cond_Extension} in combination with Proposition~\ref{Prop_RegExt} guarantees the existence of a regular extension of $g$, which is a branched covering map $\ol{g}:\S^2\to\S^2$.
Condition (\ref{Cond_Degree}) and the Riemann-Hurwitz formula ensure
that $\ol{g}$ has degree $d_{\Gamma}$.

Another way to characterize $\Gamma$ is to observe that $\Gamma$ is the component of $\ovl g^{-N_\Gamma}(\Deltahat)$ that contains $\Deltahat$ (see the proof of Theorem~\ref{Thm_NewtonGraph} below). 

\goodbreak

\begin{definition}[Equivalent Newton graphs]
We say that two abstract Newton graphs $(\Gamma_1,g_1)$ and $(\Gamma_2,g_2)$ are \emph{equivalent} if there exists a graph homeomorphism $h\colon\Gamma_1\to\Gamma_2$ that preserves the cyclic order of edges at each vertex of $\Gamma_1$ and that sends the distinguished vertex of $\Gamma_1$ to the distinguished vertex of $\Gamma_2$. 
\end{definition}

Note that equivalence of Newton graphs is a priori not a dynamical condition. This is addressed in the following lemma. 

\begin{lemma}[Equivalent Newton graphs]
\label{Lem:EquivalentGraphs}
If $(\Gamma_1,g_1)$ and $(\Gamma_2,g_2)$ are equivalent Newton graphs, then every graph homeomorphism $h\colon \Gamma_1\to\Gamma_2$ realizing this equivalence has the property that $g_2\circ h=h\circ g_1$ on the vertex set of $\Gamma_1$, and so that $g_2\circ h$ is homotopic to $h\circ g_1$ on $\Gamma_1$ relative to the vertices.

In this case, the regular extensions $\ovl g_1$ and $\ovl g_2$ are Thurston equivalent. 
\end{lemma}
\begin{proof}
By hypothesis, there exists a graph homeomorphism $h\colon \Gamma_1\to\Gamma_2$ that preserves vertices and the cyclic order of edges at all vertices, and that respects the distinguished vertices of $\Gamma_1$ and $\Gamma_2$. We will show that, up to homotopy along the edges relative to the vertices, we have 
\begin{equation}
g_2\circ h=h\circ g_1
\label{Eq_GraphEquivalence}\;.
\end{equation}

Let $\Deltahat_1\subset\Gamma_1$ and $\Deltahat_2=h(\Deltahat_1)\subset\Gamma_2$ be the abstract channel diagrams. These consist exactly of the distinguished vertices of $\Gamma_1$ and $\Gamma_2$ together with all adjacent edges, the endpoints of which are fixed points of $g_1$ resp.\ $g_2$. This implies \eqref{Eq_GraphEquivalence} on all vertices of $\hat\Delta_1$, and up to isotopy relative to the vertices on all edges.

Now let $\hat\Gamma_1\subset\Gamma_1$ be the subgraph on which \eqref{Eq_GraphEquivalence} holds for all vertices, and on all edges up to isotopy relative to the vertices. Then $\hat\Delta_1\subset\hat\Gamma_1$. If \eqref{Eq_GraphEquivalence} holds on any edge, then also at its adjacent two vertices. It thus suffices to prove that if it holds at some vertex $v$, then also at all adjacent edges. Moreover, we may by induction assume that either $v\in\hat\Delta_1$ or that \eqref{Eq_GraphEquivalence} already holds at $g_1(v)$. The inclusion $\Deltahat_1 \subset \hat \Gamma_1$ allows us to start this induction, and condition (\ref{Cond_Preimages}) allows us to run it.

If $v\in\Deltahat_1$, then either $v$ is the distinguished vertex and \eqref{Eq_GraphEquivalence} is obvious because $h$ must respect the channel diagram, or $v$ is not distinguished and all edges at $v$ and $h(v)$ are iterated preimages of edges at $v$ and $h(v)$ in $\Deltahat_1$ and $\Deltahat_2$, respectively, and the claim follows by induction on the number of preimages. Finally, if \eqref{Eq_GraphEquivalence} holds at $g_1(v)$, then condition~\eqref{Cond_Saturated} determines the number of edges at $v$ and $h(v)$, and this determines $h$ at all edges at $v$ because $h$ must preserve the cyclic order.

It remains to show the claim on Thurston equivalence.
Both graph maps $g_1$ and $g_2$ have regular extensions $\ovl g_1$ and $\ovl g_2$ (see the remark after Definition~\ref{Def_NewtonGraph}), and these extensions are compatible with the homotopy from $g_2\circ h$ to $h\circ g_1$ on $\Gamma_1$ relative to the vertices: every component of $\S^2\sm\Gamma_1$ and of $\S^2\sm\Gamma_2$ is a topological disk, and the homotopies on their boundaries extend to the disks by the Alexander trick. This means that $\ovl g_1$ and $\ovl g_2$ are Thurston equivalent (note that the existence of a homotopy is equivalent to the existence of an isotopy; see for instance \cite[Appendix~C.3]{Hubbard2}).
\end{proof}

Now we are ready to prove the first of our two theorems on the classification of postcritically fixed Newton maps, saying that every postcritically fixed Newton map gives rise to an abstract Newton graph. Recall that for a Newton map $f$ with channel diagram $\Delta$, we denote by $\Delta_n$ the component of $f^{-n}(\Delta)$ that contains $\Delta$.

\begin{proof}[Proof of Theorem \ref{Thm_NewtonGraph}]
Let $f$ be a postcritically fixed  Newton map with channel diagram $\Delta$. First observe that $\Delta$ connects every fixed point of $f$ to $\infty$. Since $f$ is postcritically fixed, each critical point of $f$ is either a prepole or an iterated preimage of a root, so each critical point that is not a prepole is connected to a prepole by an iterated preimage of a single edge in $\Delta$. Since every prepole is in some $\Delta_n$ by Corollary~\ref{Cor:PrepolesConnect}, there exists  a minimal $N\in\N$
such that $\Delta_{N-1}$ contains all critical points of $f$. We prove that $(\Delta_N,f)$ is an abstract Newton graph in the sense of Definition \ref{Def_NewtonGraph}.

Most conditions are easily seen to be satisfied. To begin with, condition (\ref{Cond_Diagram}) is Lemma~\ref{Lem:AbstractChannelDiagram}. Condition (\ref{Cond_Branch}) is also immediate: since $\infty$ is a repelling fixed point and thus has local mapping degree $1$, all edges in every $\Delta_n$ that terminate at $\infty$ are already in $\Delta$; since all other vertices of $\Delta$ are critical fixed points, they have more edges in every $\Delta_n$ with $n\ge 1$ than in $\Delta=\Delta_0$. The number of accesses to $\infty$ of every fixed point $\xi\in\C$ is one less than the degree of $f$ on $U_\xi$, and since $f$ is postcritically fixed, this degree is the local degree of $f$ at $\xi$. 

Since each $\Delta_{n+1}$ is a component of $f^{-1}(\Delta_n)$, it follows that for each vertex $v\in\Delta_n$ the local degree $\deg_v f$ in the graph map $f \colon \Delta_N \to \Delta_N$ equals the degree of $f$ as a map of $\S^2$. But $N$ was chosen such that $\Delta_{N-1}$ contains all critical points. Therefore condition (\ref{Cond_Degree}) follows from the Riemann--Hurwitz formula. 

Conditions \eqref{Cond_Enter} and \eqref{Cond_Preimages} are built into the definition of $\Delta_N$ with $N=N_\Gamma$, and the same holds for \eqref{Cond_Saturated}. 
Condition (\ref{Cond_Extension}) follows from the fact that all critical points of $f$ are in $\Delta_{N}$, and since all complementary components are simply connected, they are mapped forward injectively. 

Finally, Condition (\ref{Cond_Connected}) is implied by the following lemma that we prove below.

\begin{lemma}[Newton graph connected in $\C$]
\label{Lem:GraphConnectedInC}
For every postcritically fixed Newton map and $n$ large enough so that $\Delta_n$ contains all critical points, the graph $\ovl{\Delta_{n'}\sm\Delta}$ is connected for all $n'>n$. 
\end{lemma}

Once we prove this lemma (below), this shows that every postcritically fixed Newton map gives rise to an abstract Newton graph as in Definition~\ref{Def_NewtonGraph}.

The final claim of the theorem is that two postcritically fixed Newton maps have equivalent Newton graphs if and only if these maps are affinely conjugate. To see this, suppose that $f_1$ and $f_2$ are two postcritically fixed Newton maps with equivalent Newton graphs $(\Delta_{1,N},f_1)$ and $(\Delta_{2,N},f_2)$. By Lemma~\ref{Lem:EquivalentGraphs}, this implies that $f_1$ and $f_2$ are Thurston equivalent, and by Thurston uniqueness  (see \cite{DH} or Theorem \ref{Thm_Thurston} below), $f_1$ and $f_2$ are conjugate by a M\"obius transformation that fixes $\infty$, so they are affinely conjugate.

The converse claim that affinely equivalent Newton maps have equivalent graph maps is obvious. 
\end{proof}

Note that if
all critical points are in $\Delta_{N-1}$, we need to pull back one more time to ensure that all critical points have enough branches (condition~\eqref{Cond_Saturated}), and also that $\ovl{\Delta_N\sm\Delta}$ is connected (condition \eqref{Cond_Connected}). The latter is a consequence of Lemma~\ref{Lem:GraphConnectedInC} which we will now prove.

\begin{proof}[Proof of Lemma~\ref{Lem:GraphConnectedInC}]
Suppose by way of contradiction that
the bounded set $\ol{\Delta_{N-1}\setminus\Delta}$ is not connected, i.e.\ that $\infty$ disconnects $\Delta_{N-1}$.
Then there exists an unbounded component $V$ of
$\C\setminus\Delta_{N-1}$ that separates the plane, i.e.\ $V$ has
at least two accesses to $\infty$. Let $W$ be a neighborhood of
$\infty$ that is a round disk in linearizing coordinates and
satisfies $W\cap \Delta_{N-1}\subset \Delta$. Let $V_1,\dots,V_k$ be
the unbounded components of $V\cap W$; we have $k\ge 2$ by hypothesis (it may well be that all components of $V\cap W$ are unbounded). Then $f$ acts injectively on each
$V_i$ and there exists a branch $g_i$ of $f^{-1}$ that maps $V_i$
into itself (recall that $\infty$ is a repelling fixed point of $f$,
so it is attracting for the $g_i$). By construction, $V$ is simply
connected and contains no critical values of $f$, so the $g_i$
extend to all of $V$ by analytic continuation. Since
$\Delta_{N-1}\subset\Delta_{N}$, we obtain $g_i(V)\subset V$ for all
$i$. If there are $i\neq j$ such that $g_i(V)\cap
g_j(V)\neq\emptyset$, then it follows that $g_i=g_j$ because $g_i$ and $g_j$ are two inverse branches of $f$, but this is impossible (we have a holomorphic self-map of $V$ for which $\infty\in \partial V$
is attracting through two distinct accesses, and this contradicts for instance the Denjoy-Wolff theorem \cite[Theorem 5.4]{Milnor} or simply contraction of the hyperbolic metric).

Therefore, $g_i(V)$ are pairwise disjoint and thus all have exactly one access to $\infty$. If $w\in \partial
g_i(V)$ for some $i$, then $f(w)\in\partial V\subset \Delta_{N-1}$, for otherwise the
map $g_i$ would be defined in a neighborhood of $f(w)$. Hence $w\in
f^{-1}(\Delta_{N-1})$; since all $g_i(V)$ are open disks and have connected boundary, we have $w\in \Delta_N$: this means that for every access of $V$ to $\infty$, the two adjacent edges can be connected through $\Delta_N\cap\C$, so no component of
$\Cc\setminus \Delta_N$ has more than one access to $\infty$, and
$\ol{\Delta_N\setminus\Delta}$ is connected.
\end{proof}

\section{Abstract Newton Graphs Are Realized}
\label{Sec_Thurston}

In this section, we prove that every abstract Newton graph is realized by a postcritically fixed Newton map that is unique up to affine conjugation. This will be accomplished with the help of Thurston's fundamental characterization theorem of rational maps. We start by reviewing  some fundamental notions of Thurston theory, including the helpful notion of arc systems by Kevin Pilgrim and Tan Lei.

\subsection{Thurston's Criterion For Marked Branched Coverings}

Thurston's theorem provides a necessary and sufficient condition for the existence of a rational map with certain combinatorial behavior in terms of combinatorial mapping properties of (potentially very large) collections of simple closed curves.

The notations and results in this section are based on \cite{DH} and \cite{PT}.
Before we can state Thurston's criterion, we need several definitions. Recall that $P_g$ stands for the postcritical set of a branched covering $g:\S^2\to\S^2$.

\begin{definition}[Marked branched covering; \cite{PT}]
A \emph{marked branched covering} is a pair $(g,X)$, where $g:\S^2\to\S^2$ is a postcritically finite branched covering map and $X$ is a finite set containing $P_g$ such that $g(X)\subset X$.
\end{definition}

\begin{definition}[Thurston equivalence]
\label{Def:ThurstonEquivalence}
Let $(g,X)$ and $(h,Y)$ be two marked branched coverings. We say that they are \emph{Thurston equivalent} if there are two  orientation-preserving homeomorphisms $\phi_0,\phi_1:\S^2\to\S^2$ such that
\[
    \phi_0\circ g = h\circ \phi_1,
\]
and there exists an isotopy $\Phi: [0,1]\times\S^2\to\S^2$ with $\Phi(0,.)=\phi_0$ and $\Phi(1,.)=\phi_1$ such that $\Phi(t,.)|_{X}$ is constant in $t\in [0,1]$ with $\Phi(t,X)=Y$.
\end{definition}

\begin{definition}[Multicurve]
Let $(g,X)$ be a marked branched covering. By a \emph{simple closed curve in $(\S^2,X)$} we mean a simple closed curve $\gamma\subset\S^2\sm X$. It is called \emph{essential} if both components of $\S^2\setminus\gamma$ contain at least two points of $X$, and  \emph{peripheral} otherwise.

Two simple closed curves $\gamma_0,\gamma_1$ in $(\S^2,X)$ are called \emph{isotopic (relative $X$)} (write $\gamma_0\simeq\gamma_1$) if there exists a continuous one-parameter family $\gamma_t$ (with $t\in [0,1]$) of curves in $(\Sphere, X)$ joining $\gamma_0$ to $\gamma_1$. We denote the isotopy class of $\gamma_0$ by $[\gamma_0]$.

A finite set $\Pi=\{\gamma_1,\dots,\gamma_m\}$ of disjoint, essential and pairwise non-isotopic simple closed curves in $(\S^2,X)$ is called a \emph{multicurve}.
\end{definition}

If $(g,X)$ is a marked branched covering and $\gamma$ is a simple closed curve in $\S^2\setminus X$, then the set $g^{-1}(\gamma)$ is a disjoint union of simple closed curves. 

\begin{definition}[Irreducible Thurston (multicurve) obstruction]
\label{Def:IrThObs}
Let $(g,X)$ be a marked bran\-ch\-ed covering and $\Pi$ a multicurve. Denote by $\R^\Pi$ the real vector space spanned by the isotopy classes of the curves in $\Pi$. Then we associate to $\Pi$ its \emph{Thurston transformation} $g_\Pi:\R^\Pi\to\R^\Pi$ by specifying its action on representatives $\gamma\in\Pi$ of basis elements:
\begin{equation}
\label{Eqn_ThurstonTransform}
    g_{\Pi}(\gamma) := \sum_{\gamma'\in g^{-1}(\gamma)} \frac{1}{\deg(g|_{\gamma'}:\gamma'\to\gamma)}[\gamma']
\end{equation}
where the sum is taken over all preimage components $\gamma'$ of $\gamma$ that are essential curves in $(\S^2,X)$ and isotopic to one of the curves in $\Pi$; the sum is zero if there are no such components. 

Note that in Definition~\ref{Def:IrThObs} we do not assume that $\Pi$ is invariant; one can always extend to the invariant case, see \cite[Remark~2 after Theorem 3.1]{PT}. Note also that it is quite possible for to different $\gamma_1,\gamma_2\in\Pi$ (which are non-homotopic by definition) to have essential preimage curves $\gamma'_1\in g^{-1}(\gamma_1)$, $\gamma'_2\in g^{-1}(\gamma_2)$ that are homotopic, i.e.\ with $[\gamma'_1]=[\gamma'_2]$.

The linear transformation $g_\Pi$ given by (\ref{Eqn_ThurstonTransform}) is represented by a square matrix with non-negative entries. Therefore its largest eigenvalue $\lambda(\Pi)$ is real and non-negative by the Perron-Frobenius theorem.  A multicurve $\Pi$ is called a \emph{Thurston obstruction} if $\lambda(\Pi)\ge 1$; we prefer to use the term \emph{multicurve obstruction}.

We will focus on irreducible obstructions; these are defined as follows. 
A square matrix $A$ with non-negative entries is called \emph{irreducible} if for each pair $(i,j)$ there exists $k \geqslant 0$ such that the $(i,j)$ entry of $A^k$ is positive. We say that $\Pi$ is an \textit{irreducible multicurve} if the matrix representing $g_\Pi$ is. 
Finally, an irreducible multicurve $\Pi$ is called an \emph{irreducible multicurve obstruction} if $\lambda(\Pi)\geqslant 1$. 
\end{definition}

The statement of Thurston's theorem contains the notion of a hyperbolic orbifold. For our Newton maps the orbifolds are always hyperbolic, so we refrain from giving a precise definition (which is given in \cite{DH}); suffice it to say that whenever a branched cover has at least three critical fixed points, it has hyperbolic orbifold, and this is the case for postcritically finite Newton maps of degree at least $3$, hence with at least $3$ roots.

Now we are ready to state Thurston's theorem for marked branched coverings as given in \cite[Theorem 3.1]{PT} and proved in \cite{DH}.

\begin{theorem}[Marked Thurston theorem]
\label{Thm_Thurston}
Let $(g,X)$ be a marked branched covering with hyperbolic orbifold. It is Thurston equivalent to a marked rational map $(f,Y)$ if and only if it does not have an irreducible multicurve obstruction, that is if and only if $\lambda(\Pi)<1$  for every irreducible multicurve $\Pi$. In this case, the rational map $f$ is unique up to M\"obius  conjugation.
\qed
\end{theorem}

\subsection{Arcs Intersecting Obstructions}

We present a theorem of Ke\-vin Pil\-grim and Tan Lei \cite{PT} that is useful to show that certain marked branched coverings do not have Thurston obstructions, so they are equivalent to rational maps. Again, we first need to introduce some notation.

Let $(g,X)$ be a marked branched covering of degree $d \geqslant 3$.
\begin{definition}[Arc system]
An \emph{arc in $(\S^2,X)$} is a map $\alpha:[0,1]\to \S^2$ such
that $\alpha(\{0,1\})\subset X$ and $\alpha((0,1))\cap X=\emptyset$ and so that $\alpha$ is a continuous mapping that is injective on $(0,1)$. 
For two arcs $\alpha_1$ and $\alpha_2$ we write $\alpha_1\simeq \alpha_2$ if they are isotopic relative $X$.

A set of pairwise non-isotopic arcs in $(\S^2,X)$ is called an \emph{arc system}. Two arc systems $\Lambda,\Lambda'$ are \emph{isotopic} if each curve in $\Lambda$ is isotopic relative $X$ to a
unique element of $\Lambda'$ and vice versa.
\end{definition}

Note that arcs connect marked points (the endpoints of an arc need not be distinct) while
simple closed curves run around them. The resulting
intersections will give us some control over the
location of possible multicurve obstructions. Since arcs and curves are only
defined up to isotopy, we make precise what we mean by intersecting arcs and
curves.

\begin{definition}[Intersection number]
\label{Def_IntersectionNumber}
Let $\alpha$ and $\beta$ each be an arc or a
simple closed curve in $(\S^2,X)$. Their \emph{intersection
number} is
\[
    \alpha\cdot\beta := \min_{\alpha'\simeq\alpha,
    \,\beta'\simeq\beta} \#((\alpha'\cap\beta')\setminus X).
\]
The intersection number extends bilinearly to arc systems and multicurves.
\end{definition}

If $\alpha$ is an arc in $(\S^2,X)$, then the closure of a component of $g^{-1}(\alpha\sm X)$ is called a \emph{lift}
of $\alpha$. Each arc clearly has $d$ distinct lifts. If $\Lambda$ is an arc system, an arc system $\tilde{\Lambda}$ is
called a \emph{lift} of $\Lambda$ if each $\tilde{\alpha}\in\tilde{\Lambda}$ is a lift of some
$\alpha\in\Lambda$. 

If $\Lambda$ is an arc system, we introduce a linear map $g_{\Lambda}$ on the real vector space $\R^{\Lambda}$ similarly as for multicurves: for $\alpha\in\Lambda$, set
\[
    g_{\Lambda}(\alpha):=\sum_{\alpha' \in g^{-1}(\alpha)} [\alpha']\;,
\]
where $[\alpha']$ denotes the isotopy class of $\alpha'$ relative $X$ and the sum is again over all $\alpha' \in g^{-1}(\alpha)$ that are in $\Lambda$. Again, the sum is assumed to be zero if $\alpha$ has no preimages that are isotopic to elements of $\Lambda$. We say that $\Lambda$ is \emph{irreducible} if the matrix representing $g_{\Lambda}$ is.

Denote by $\tilde{\Lambda}(g^{\circ n})$ the union of those components of $g^{-n}(\Lambda)$ that are isotopic to elements of $\Lambda$ relative $X$, and define $\tilde{\Pi}(g^{\circ n})$ analogously. Note that if $\Lambda$ is irreducible, each
element of $\Lambda$ is isotopic to an element of $\tilde{\Lambda}(g^{\circ n})$ for some $n$.

The following theorem is \cite[Theorem~3.2]{PT}. It shows that, up to isotopy, irreducible multicurve obstructions cannot intersect the preimages of irreducible arc systems, except possibly the arc systems themselves. We will use this theorem to show that the extended map of every abstract Newton graph is Thurston equivalent to a rational map.

\begin{theorem}[Arcs intersecting obstructions] 
\label{Thm_ArcsInterOb} 
Let $(g,X)$ be a mar\-ked branched covering, $\Pi$ be an irreducible multicurve obstruction and $\Lambda$ be an
irreducible arc system. Suppose furthermore that $\#(\Pi\cap\Lambda)=\Pi\cdot\Lambda$. Then exactly one of
the following is true:
\begin{enumerate}
\item 
$\Pi\cdot\Lambda=0$ and $\Pi\cdot g^{-n}(\Lambda)=0$ for all $n \geqslant 1$.
\item 
$\Pi\cdot\Lambda\neq 0$ and for each $n \geqslant 1$ each component of $\Pi$ is isotopic to a unique
component of $\tilde{\Pi}(g^{\circ n})$. The mapping $g^{\circ n} \colon \tilde{\Pi}(g^{\circ n})\to \Pi$ is a homeomorphism and
$\tilde{\Pi}(g^{\circ n})\cap \left(g^{-n}(\Lambda)\setminus
\tilde{\Lambda}(g^{\circ n})\right)=\emptyset$. 
\end{enumerate}
The same is true when interchanging the roles of $\Pi$ and $\Lambda$. \qed
\end{theorem}

\subsection{The Realization of Abstract Newton Graphs}

Our last proof shows that every abstract Newton graph is realized by a postcritically fixed Newton map, unique up to affine conjugation.

\begin{proof}[Proof of Theorem \ref{Thm_Realization}]
Let $(\Gamma,g)$ be an abstract Newton graph in the sense of Definition~\ref{Def_NewtonGraph} and let  $\ovl g$ be its regular extension to $\S^2$; it always exists and is a branched self-covering of $\Sphere$ (using Definition~\ref{Def_NewtonGraph}, condition~\eqref{Cond_Extension}; see the remark after that definition) and has degree $d_\Gamma$ (condition \eqref{Cond_Degree}). Let $X:=\Gamma'$ be the set of vertices of $\Gamma$. Let $\Deltahat\subsetneq\Gamma$ be the abstract channel diagram of $\Gamma$ and let $v_0,v_1,\dots,v_{d_\Gamma}$ be the vertices of $\hat\Delta$ so that $v_0$ is the distinguished vertex; these exist by Definition~\ref{Def_NewtonGraph}, condition~\eqref{Cond_Diagram}. 
By condition (\ref{Cond_Branch}), the vertices $v_1,\dots,v_{d_\Gamma}$ of $\Deltahat$ are critical points of $\ol{g}$. Since $d_{\Gamma}\geqslant 3$, the orbifold of $\ol{g}$ is hyperbolic and it suffices to show that $(\ol{g},X)$ has no irreducible multicurve obstruction: it then follows from Theorem \ref{Thm_Thurston} that $\ol{g}$ is Thurston equivalent to a rational map $f$ of degree $d_\Gamma$, and the latter is unique up to M\"obius transformation. Then $f$ has $d_\Gamma+1$ fixed points, $d_\Gamma$ of which are superattracting because $\ol{g}$ has marked fixed critical points $v_1,\dots,v_{d_\Gamma}$. The last fixed point $v_0$ must then be repelling \cite[Corollaries 12.7 and 14.5]{Milnor} and after possibly conjugating $f$ with a M\"obius transformation, we may assume that it is at $\infty$. Now it follows from Proposition~\ref{Prop_Head} that $f$ is a Newton map. It is unique up to a M\"obius transformation fixing $\infty$, hence up to affine conjugacy.

The maps $f$ and $\ovl g$ are Thurston equivalent; this is equivalent to the fact that $\ovl g$ can be precomposed with an isotopy relative to the marked points so that $f$ and $\ovl g$ are topologically conjugate. Since $\ovl g$ is defined in its construction only up to such an isotopy, we may as well assume that $f$ and $\ovl g$ are topologically conjugate. This implies in particular that there is a graph homeomorphism between $\Deltahat$ and the channel diagram of $f$ (the image of the channel diagram of $f$ is a channel diagram of $\ovl g$, and any additional edge in the channel diagram of $\ovl g$ would yield a curve that connects a fixed point in $\C$ of $f$ to $\infty$ that is homotopic to its image, and every such curve is already homotopic to an invariant access of $f$). By taking backwards images it follows that we even have a graph homeomorphism between $(\Gamma,g)$ and the Newton graph of $f$ that is compatible with the dynamics, as in the proof of Lemma~\ref{Lem:EquivalentGraphs} (note that $(\Gamma,g)$ contains all information about critical points of $\ovl g$), so $(\Gamma,g)$ is realized by $f$.

It remains to prove the claim that $(\ovl g,X)$ indeed does not have an irreducible multicurve obstruction. Suppose by way of contradiction that $\Pi$ is an irreducible multicurve obstruction for $(\ol{g},X)$ and let $\gamma\in\Pi$. Then $\gamma$ is an essential simple closed curve in $\S^2\setminus X$. Each edge $\alpha$ of $\Deltahat$ forms an irreducible arc system, so Theorem \ref{Thm_ArcsInterOb} implies (in both cases stated) that $\gamma\cdot (\ol{g}^{-n}(\alpha)\setminus\alpha)=0$ for all $n \geqslant 1$. Since this is true for all edges of $\Deltahat$ and each edge in $ \ovl{\Gamma\sm\Deltahat}$ is an iterated preimage of $\Deltahat$ (condition~\eqref{Cond_Preimages} of Definition~\ref{Def_NewtonGraph}), it follows that $\gamma\cdot \left(\ovl{\Gamma\sm\Deltahat}\right)=0$. But since $\ol{\Gamma\setminus\Deltahat}$ is connected (condition \eqref{Cond_Connected}) and contains $X\setminus\{v_0\}$, this means that $\gamma$ is peripheral, a contradiction. Therefore, $(\ovl g,X)$ does not have a multicurve obstruction and is thus Thurston equivalent to a postcritically fixed Newton map.

The last claim in Theorem~\ref{Thm_Realization} concerns the case that $f$ realizes two abstract Newton graphs $(\Gamma_1,g_1)$ and $(\Gamma_2,g_2)$. By definition, this means that $(\Gamma_1,g_1)$ and $(\Gamma_2,g_2)$ are graph equivalent to $f$ with its Newton graph, and hence to each other. 
\end{proof}

In conclusion, Theorem~\ref{Thm_NewtonGraph} provides a map ``Newton $\to$ Graphs'' from postcritically fixed Newton maps (modulo affine conjugacy) to abstract Newton graphs (modulo graph equivalence), and Theorem~\ref{Thm_Realization} provides a map ``Graphs $\to$ Newton'' in the opposite direction that is an inverse to the ``Newton $\to$ Graphs'' map; this proves that the map ``Newton $\to$ Graphs'' injective and the ``Graphs $\to$ Newton'' map is surjective. By Theorem~\ref{Thm_Realization}, the map ``Graphs $\to$ Newton'' is also injective and hence bijective, and the map ``Newton $\to$ Graphs'' is its inverse. 

This concludes the classification of postcritically fixed Newton maps in terms of abstract Newton graphs. As mentioned earlier, this result has been extended in \cite{LMS1,LMS2} to a classification of postcritically \emph{finite} Newton maps: in order to treat critical orbits that terminate in cycles of length two or greater, one needs to treat embedded periodic Hubbard trees (so that classification builds upon a classification of polynomial Hubbard trees).


\end{document}